\documentclass[a4paper,reqno]{amsart}
\usepackage[all]{xy}           
\usepackage{amssymb}           
\usepackage{eucal}
\usepackage[usenames,dvipsnames]{color}

\newcommand{\calf}{\mathcal{F}}

\newcommand{\Le}{\mathbb{L}�_h}

\numberwithin{equation}{section}

\newtheorem{definition}{Definition}[section]
\newtheorem{lemma}[definition]{Lemma}
\newtheorem{theorem}[definition]{Theorem}
\newtheorem{proposition}[definition]{Proposition}

\newtheorem{remarkth}[definition]{Remark}

\newenvironment{remark}{\begin{remarkth}\upshape}{\hfill$\diamond$\end{remarkth}}
\renewcommand{\emph}[1]{{\bfseries\itshape{#1}}}

\usepackage{amssymb}

\newcommand{\R}{\mathbb{R}}      
\newcommand{\F}{\mathbb{F}}

\newcommand{\T}{\mathbb{T}}

\newcount\ancho \newcount\anchom \newcount\anchoa
\newcount\anchob \newcount\altura

\newcommand{\ltilde}[3][0]{\altura=0 \advance\altura by #1
           \ancho=#2 \anchom=\ancho \divide\anchom by 2
           \anchoa=\ancho \divide\anchoa by 4
           \anchob=\anchom \advance\anchob by \anchoa
           \kern-3pt \begin{array}[b]{c}
           \begin{picture}(1,1)(\anchom,-\altura)
        \qbezier(0,2)(\anchoa,5)(\anchom,2)
        \qbezier(\anchom,2)(\anchob,-1)(\ancho,4)
        \qbezier(0,2)(\anchoa,4.5)(\anchom,1.8)
        \qbezier(\anchom,1.8)(\anchob,-1.5)(\ancho,4)
       \end{picture} \\[-4pt]{#3}
                       \end{array} \kern-4pt    }

\newcommand{\lhat}[3][0]{\altura=0 \advance\altura by #1
           \ancho=#2 \anchom=\ancho \divide\anchom by 2
           \anchoa=\ancho \divide\anchoa by 4
           \anchob=\anchom \advance\anchob by \anchoa
           \kern-3pt \begin{array}[b]{c}
           \begin{picture}(1,1)(\anchom,-\altura)
        \qbezier(0,2)(\anchoa,4)(\anchom,6)
        \qbezier(\anchom,6)(\anchob,4)(\ancho,2)
        \qbezier(0,2)(\anchoa,3.8)(\anchom,5.6)
        \qbezier(\anchom,5.6)(\anchob,3.8)(\ancho,2)
       \end{picture} \\[-4pt] {#3}
                       \end{array} \kern-4pt    }

\newcommand{\lcf}{\lbrack\! \lbrack}
\newcommand{\rcf}{\rbrack\! \rbrack}
\newcommand\map[3]{#1\ \colon\ #2\longrightarrow#3}
\newcommand{\lvec}[1]{\overleftarrow{#1}}
\newcommand{\rvec}[1]{\overrightarrow{#1}}

\newcommand{\qquand}{\qquad\text{and}\qquad}

\newcommand{\e}{\mathrm{e}}

\makeatletter
\newcommand\prol{\@ifstar{\@proldf}{\@prolpf}}  
\def\@prolpf{\@ifnextchar[{\@prolpf@wrt}{\@prolpf@}}
\def\@prolpf@wrt[#1]#2{\@ifnextchar[{\@prolpf@wrt@at{#1}{#2}}{\@prolpf@wrt@{#1}{#2}}}
\def\@prolpf@wrt@at#1#2[#3]{\prolsymbol^{#1}_{#3}#2}
\def\@prolpf@wrt@#1#2{\prolsymbol^{#1}#2}
\def\@prolpf@#1{\@ifnextchar[{\@prolpf@at{#1}}{\@prolpf@@{#1}}}
\def\@prolpf@at#1[#2]{\prolsymbol_{#2}#1}
\def\@prolpf@@#1{\prolsymbol#1}
\def\@proldf{\@ifnextchar[{\@proldf@wrt}{\@proldf@}}
\def\@proldf@wrt[#1]#2{\@ifnextchar[{\@proldf@wrt@at{#1}{#2}}{\@proldf@wrt@{#1}{#2}}}
\def\@proldf@wrt@at#1#2[#3]{\prolsymbol^{*#1}_{#3}#2}
\def\@proldf@wrt@#1#2{\prolsymbol^{*#1}#2}
\def\@proldf@#1{\@ifnextchar[{\@proldf@at{#1}}{\@proldf@@{#1}}}
\def\@proldf@at#1[#2]{\prolsymbol^*_{#2}#1}
\def\@proldf@@#1{\prolsymbol^*#1}
\def\prolsymbol{\mathcal{T}}
\makeatother

\def\lcf{\lbrack\! \lbrack}
\def\rcf{\rbrack\! \rbrack}
\newcommand{\pai}[2]{\langle #1, #2\rangle}

\newcommand{\cinfty}[1]{C^\infty(#1)}
\newcommand{\set}[2]{\left\{\,#1\left.\vphantom{#1#2}\,\right\vert\,#2\,
                \right\}}

\newcommand{\pd}[2]{\frac{\partial #1}{\partial #2}}
\newcommand{\at}[1]{\Big|_{#1}}
\newcommand{\vectorfields}[1]{\mathfrak{X}(#1)}

\newcommand{\sode}{{\textsc{sode}}}
\newcommand{\ode}{{\textsc{ode}}}

\newcommand{\Real}{\mathbb{R}}
\begin{document}
{\Large

\title[On the exact discrete Lagrangian function for variational integrators]{On the exact discrete Lagrangian function for variational integrators: theory and applications}

\author[J.C. Marrero]{J.C. Marrero}
\address{J.C. Marrero:
ULL-CSIC Geometr\'{\i}a Diferencial y Mec\'anica Geom\'etrica\\
Departamento de Matem\'aticas, Estad{\'\i}stica e IO, Secci\'on de
Ma\-te\-m\'a\-ti\-cas y F{\'\i}sica, Universidad de la Laguna, La Laguna,
Tenerife, Canary Islands, Spain} \email{jcmarrer@ull.edu.es}

\author[D. Mart\'{\i}n de Diego]{D. Mart\'{\i}n de Diego}
\address{D. Mart\'{\i}n de Diego:
Instituto de Ciencias Matem\'aticas (CSIC-UAM-UC3M-UCM) \\ C/Nicol\'as
Cabrera 13-15, 28049 Madrid, Spain} \email{david.martin@icmat.es}

\author[E. Mart\'{\i}nez]{E. Mart\'{\i}nez}
\address{E. Mart\'{\i}nez:
Departamento de Matem\'atica Aplicada e IUMA, Facultad de
Ciencias, Universidad de Zaragoza, 50009 Zaragoza, Spain}
\email{emf@unizar.es}

\thanks{{This work has been partially supported by grants MTM 2012-34478, MTM 2013-42
870-P, MTM 2015-64166-C2-2P (MINECO),  the European project IRSES-project ``Geomech-246981'' and the ICMAT Severo Ochoa projects SEV-2011-0087 and 
SEV-2015-0554 (MINECO). JCM acknowledges the partial support from IUMA (University of
Zaragoza) for a stay at the University of Zaragoza where this work was started}}

\keywords{ Exact discrete Lagrangian, variational integrators, Lie algebroids, Lie groupoids, second order differential equations, convexity theorems, variational error analysis, discrete Euler-Poincar\'e equations, discrete Lagrange-Poincar\'e equations}

\subjclass[2010]{ 17B66,
22A22, 70G45, 70Hxx.}

\begin{abstract}
In this paper, we will give a rigorous construction of the exact discrete Lagrangian formulation associated to a continuous Lagrangian problem. Moreover, we work in the  setting of Lie groupoids and Lie algebroids which is enough general to simultaneously cover several cases of interest in discrete and continuous descriptions as, for instance,  Euler-Lagrange equations, Euler-Poincar\'e equations, Lagrange-Poincar\'e equations...
The construction of an exact discrete Lagrangian is of considerable interest for the analysis of the error between an exact
trajectory and the discrete trajectory derived by a variational integrator.

\end{abstract}

\maketitle

\tableofcontents

\section{Introduction}

Classical integrators are mainly focused on solving ordinary differential equations (\ode s) on an euclidean space. In many cases of interest, without solving explicitly a given  \ode, we known  some important qualitative and geometric features of the  solutions of the equations of motion. For instance, preservation of the energy or other conservation laws, preservation of geometric structures like symplectic  forms, volume forms, Poisson structures... or even, preservation of the manifold  structure of the configuration space where the \ode\ evolves (see Marsden, Ratiu   \cite{MaRa}).

As a consequence,  in the last years,  the theme of structure preservation has emerged in numerical analysis with important ramifications in differential geometry, theoretical mechanics and engineering applications. 
The idea is to design numerical methods for a given \ode\ while preserving one or more of these geometric or qualitative  properties exactly. These methods are called geometric integrators  (see  Hairer et al  \cite{HLW}, McLachlan, Quispel \cite{MaQu} and reference therein). For example, if the \ode\ is defined on a Lie group it is possible to construct adapted geometric integrators 
(see Iserles et al \cite{IMNZ},  Celledoni et al \cite{CeMaOw}). Moreover, these preservation properties are inherent   in the case of Lagrangian and Hamiltonian systems  (see Marsden, West \cite{MaWe} or Sanz-Serna, Calvo \cite{SaCa}).

In our paper, we shall focus  on a particular family of geometric integrators, the  variational integrators as in Marsden, West \cite{MaWe} and Leok \cite{Leok}. This particular family of integrators is derived from  a discretization of the Hamilton's principle, that is, replacing the action integral associated to a continuous Lagrangian by a discrete action sum and  extremizing it over all the sequences of points with fixed end point conditions. Assuming a regularity condition, the derived numerical method is guaranteed to be symplectic or Poisson preserving and also,  under symmetry invariance, it is obtained preservation of the associated constants of the motion.

Moreover, we will study this problem using the more general and inclusive  perspective of Lie groupoids and Lie algebroids. For instance, in the seminal paper by Moser and Veselov \cite{mo-ve}, the authors study discrete mechanics for  Lagrangian systems of the form $L_d: G\rightarrow {\mathbb R}$  defined on a Lie group
$G$, and the discrete dynamical system is given by a diffeomorphism from $G$ to
itself. This  gives a discrete equivalent to the Euler-Poincar\'e equations. In this direction and following the program proposed by Weinstein \cite{We}, in Marrero et al \cite{MaMaMa} we have analyzed the geometrical framework for discrete Mechanics on
Lie groupoids covering simultaneously many of the possible cases of interest in mechanics. In that paper, we found intrinsic expressions for the discrete Euler-Lagrange
equations, and we have introduced the Poincar\'e-Cartan sections, the discrete
Legendre transformations and the discrete evolution operator in both the Lagrangian
and the Hamiltonian formalism. The notion of regularity has been completely
characterized and we have proven the symplecticity of the discrete evolution
operators.
Moreover, the general theory of discrete symmetry reduction directly follows from our
results.
Recently in Marrero et al \cite{MaMaMa2} we have derived   local expressions for the different
objects appearing in discrete Mechanics on Lie groupoids.

Obviously, the main application of discrete variational calculus is the construction of geometric integrators using different discretizations of the associated action sum for a continuous Lagrangian problem. 
For the analysis of the error it is crucial the construction of the exact discrete Lagrangian which exactly describes  the evolution of the continuous system.  
Marsden and West \cite{MaWe} considered this construction  for Lagrangians defined on the tangent bundle of a manifold and the  local
error analysis of variational integrators defined on $Q\times Q$ as the discrete space corresponding to the tangent bundle $TQ$. This result was later rigorously stablished by Patrick and Cuell \cite{PaCu} using variational arguments.

 In our paper, we derive this result using different arguments and admitting an extension to discrete Lagrangians defined on a general Lie groupoid. In particular, we use convexity theorems for explicit second order differential equations to construct the exact discrete Lagrangian associated to a continuous regular Lagrangian. We start with the case  of an explicit  second order differential equation defined on a configuration space $Q$. Geometrically, it is represented by a special vector field on the tangent bundle $TQ$ which is called a \sode. Roughly speaking it is possible to show  Hartman \cite{Ha} that given any two enough close  points we can find a unique solution of the second order differential equation joining both points (see Theorem  \ref{convexity1}). This  result allows us to introduce the notion of exponential map associated with a given \sode. In Theorem \ref{convexity2} we prove that this exponential map is a local diffeomorphism. We show in Theorem \ref{convexity-definitivo}  how to extend this result  to the more general case of integrable Lie algebroids, viewing  the tangent bundle as a particular case.  
With all these previous ingredients it is possible to define the exact discrete Lagrangian for a given regular Lagrangian $L: AG\rightarrow {\mathbb R}$, where $AG$ is the associated Lie algebroid to a Lie groupoid $G$ (see Marrero et al \cite{MaMaMa}), as follows
\[
\mathbb{L}_h^{e}(g) = \int_0^h L(\Phi^{\Gamma_L}_t(R^{e^-}_h(g)) dt, \;\; \mbox{ for } g \in U.
\]
Here, $g$ is defined on an appropriate open subset $U$ of $G$ near of the identities,  $\Phi^{\Gamma_L}_t$ is the flow at time $t$ associated to the \sode\ $\Gamma_L$,  which defines the corresponding Euler-Lagrange equations (see Section \ref{Con-Lag-Mech-Lie-alg}), and $R^{e^-}_h$ is the inverse of the exponential map of $\Gamma_L$. This definition agrees with the one given by Marsden and West \cite{MaWe} for the case of Lagrangians defined on tangent bundles.

Using $\mathbb{L}_h^{e}: U\subset G\rightarrow {\mathbb R}$ as a discrete Lagrangian, and following  Marrero et al  \cite{MaMaMa} (see  also Section \ref{discrete-mechanics}), we can construct the corresponding discrete Lagrangian evolution operator which defines the discrete Euler-Lagrange equations on a Lie groupoid and  the associated discrete Legendre transformations $\mathbb{F}^+\mathbb{L}^{e}_h$ and   $\mathbb{F}^-\mathbb{L}^{e}_h$.  With both Legendre transformations it is possible to define the discrete Hamiltonian evolution operator which, in the case of the exact discrete Lagrangian, is just the flow at fixed time of  the Hamiltonian vector field associated with the Lagrangian $L$ (see Theorem \ref{exact-discrete-hamiltonian-flow}).

Now rather than considering how closely the trajectory of an arbitrary numerical method  matches the exact trajectory given in our case by  $\Phi^{\Gamma_L}_t$,
we can alternatively study how closely a discrete Lagrangian matches the exact discrete Lagrangian $\mathbb{L}_h^{e}$. That is, we study the variational error analysis (see Marsen, West   
\cite{MaWe}). In Theorem \ref{theorem-error} we show that if the we take as a discrete Lagrangian an approximation of order $r$ of the exact discrete Lagrangian then the associated  discrete evolution operator is also of order $r$, that is, the derived  discrete scheme is an approximation of  the continuous flow of order $r$. 

As we mentioned before, in the particular case when $G$ is the pair groupoid $Q \times Q$, the previous results are well-known (see Marsden, West \cite{MaWe} and Patrick, Cuell \cite{PaCu}). However, the scope of our results is wider and, in fact, we apply them to two interesting problems: i) the discretization of Euler-Poincar\'e equations for continuous Lagrangians which are defined on the Lie algebra of a Lie group and ii) the discretization of the Lagrange-Poincar\'e equations for continuous Lagrangians defined on the Atiyah algebroid associated with a trivial principal bundle. 
Additionally, we compare our results with other approximations as in Bogfjellmo, Marthinsen \cite{BoMa}, Bou-Rabee, Marsden \cite{Rabee}, Leok, Shingel \cite{LeSh}, Marsden et al \cite{MaPeSh} among others.

The paper is structured as follows. In Section \ref{sec2}, we review some constructions on continuous and discrete Lagrangian mechanics on Lie algebroids and Lie groupoids, respectively. In Section \ref{section2}, we prove some convexity theorems for standard second order differential equations on smooth manifolds and, in Section \ref{sec4}, we extend these results to the Lie algebroid setting. In Section \ref{sec5}, we introduce the exact discrete Lagrangian function associated with a regular continuous Lagrangian function on a Lie algebroid, we discuss the variational error analysis and we apply the corresponding results to several interesting examples. The paper ends with two appendices which contain the proof of Theorem \ref{convexity1} and some basis constructions on Lie algebroids and groupoids.

\section{Continuous and discrete Lagrangian Mechanics on Lie algebroids and groupoids}\label{sec2}

\subsection{Continuous Lagrangian Mechanics on Lie algebroids}\label{Con-Lag-Mech-Lie-alg}

In this section, we will present a brief description of Lagrangian Mechanics on Lie algebroids (for more details, see \cite{CoLeMaMaMa,LeMaMa,Ma0,We}). 

Let $\tau: A \to M$ be a vector bundle over a manifold $M$ endowed with a Lie algebroid structure:
\[
\lcf \cdot, \cdot \rcf: \Gamma(A) \times \Gamma(A) \to \Gamma(A), \; \; \rho: A \to TM.
\]
This means that $\lcf \cdot, \cdot \rcf $ defines a Lie algebra structure on the space of sections $\Gamma(A)$ of $A$ and that $\rho: A \to TM$ is a bundle map, the anchor map, satisfying the condition
\[
\lcf X, fY \rcf = f \lcf X, Y\rcf + \rho(X)(f) Y
\]
for $X, Y \in \Gamma(A)$ and $f \in C^{\infty}(M)$ (see \cite{Mac}).

Suppose that $(x^{i})$ are local coordinates on $M$ and that $\{e_{\alpha}\}$ is a local basis of sections of $A$ such that
\[
\lcf e_{\alpha}, e_{\beta} \rcf = C_{\alpha \beta}^{\gamma}e_{\gamma}, \; \; \rho (e_{\alpha}) = \rho^{i}_{\alpha} \frac{\partial}{\partial x^{i}}.
\]
$C_{\alpha \beta}^{\gamma}$ and $\rho^{i}_{\alpha}$ are the local structure functions of $A$ for the local coordinates $(x^{i})$ and the basis $\{e_{\alpha}\}$.

Moreover, we will denote by $(x^{i}, y^{\alpha})$ (resp., $(x^{i}, y_{\alpha})$) the corresponding local coordinates on $A$ (resp., on the dual bundle $A^*$ to $A$).

On the dual bundle $A^*$ to $A$ one may define a linear Poisson bracket $\{\cdot, \cdot\}$ which is characterized by the following conditions
\[
\{\hat{X}, \hat{Y}\} = -\widehat{\lcf X, Y\rcf}, \; \; \{\hat{X}, g\circ \tau^*\} = - \rho(X)(g) \circ \tau^*, \; \; \{f \circ \tau^*, g \circ \tau^*\} = 0,
\]
for $X, Y \in \Gamma(A)$ and $f, g \in C^{\infty}(M)$. Here, $\tau^*: A^* \to M$ is the vector bundle projection and if $Z \in \Gamma(A)$ then $\hat{Z}: A^* \to \mathbb{R}$ is the fiberwise linear function on $A^*$ given by
\[
\hat{Z}(\mu) = <\mu, Z(\tau^*(\mu))>, \; \; \mbox{ for } \mu \in A^*.
\]
Note that
\[
\{y_{\alpha}, y_{\beta}\} = -C_{\alpha \beta}^{\gamma}y_{\gamma}, \; \; \{y_{\alpha}, x^{i}\} = -\rho^{i}_{\alpha}, \; \; \{x^{i}, x^{j}\} = 0,
\]
(for more details, see \cite{CoDaWe,Co}).

Now, let $L: A \to \mathbb{R}$ be a Lagrangian function on $A$. 

Then, one may introduce the Legendre transformation associated with $L$
\[
{\mathcal F}{L}: A \to A^*, \; \; \; a \in A \to {\mathcal F}{L}(a)(a') = \frac{d}{dt}\Big|_{t=0} L(a + ta').
\]
The local expression of ${\mathcal F}_{L}$ is
\[
{\mathcal F}{L}(x^{i}, y^{\alpha}) = (x^{i}, \frac{\partial L}{\partial y^{\alpha}}).
\]
The Lagrangian function $L$ is said to be (hyper)-regular if ${\mathcal F}L$ is a (global) local diffeomorphism.
In the case of hyper-regular Lagrangians, one may consider the Hamiltonian function $H: A^*\rightarrow {\mathbb R}$
\[
H = E_L \circ {\mathcal F}{L}^{-1}.
\]
Here, $E_L$ is the Lagrangian energy given by
\[
E_L = \Delta(L) - L,
\]
where $\Delta$ is the Liouville vector field on $A$. The local expression of $E_L$ is
\[
E_L = y^{\alpha} \frac{\partial L}{\partial y^{\alpha}} - L.
\]
The Hamiltonian vector field $X_H$ on $A^*$, with Hamiltonian function $H: A^*\rightarrow {\mathbb R}$, is given as follows
\[
X_H(F) = \{F, H\}, \; \; \; \mbox{ for } F\in C^{\infty}(A^*).
\]
$X_H$ induces a second order differential equation (\sode) on $A$, which we denote by $\Gamma_{L} \in {\frak X}(A)$ by
\[
(T_a {\mathcal F}{L})(\Gamma_L(a)) = X_H(Leg_L(a)), \; \;  \hbox{for  } a \in A
\]
 (see \cite{LeMaMa}). The second order condition means that the integral curves of $\Gamma_L$ are admissible, that is, if $c: I \to A$ is an integral curve of $\Gamma_L$ then
\[
\rho \circ c = \frac{d}{dt}(\tau \circ c).
\]
In general a \sode\ $\Gamma$ on $A$ is defined as a vector field $\Gamma \in {\mathfrak X}(A)$ such that $T\tau\circ \Gamma=\rho$. In local coordinates, their integral curves
\[
c: I \to A, \; \; \; t \in I \to c(t) = (x^{i}(t), y^{\alpha}(t))
\]
 satisfy the following system of differential equations: 
\[
\displaystyle \frac{dx^{i}}{dt} = \rho^{i}_{\alpha}y^{\alpha}, \quad \frac{dy^{\alpha}}{dt} = \Gamma^{\alpha}(x, y)\; .
\]
where $\Gamma=\displaystyle \rho^i_{\alpha}y^{\alpha}\frac{\partial}{\partial x^i}+\Gamma^{\alpha}(x, y)\frac{\partial}{\partial y^{\alpha}}$.

In the case of $\Gamma_L$,  its  integral curves  are just the solutions of the Euler-Lagrange equations for $L$. In fact, a curve $c$ in $A$
is an integral curve of $\Gamma_L$ if and only if
\begin{equation}\label{Euler-Lagrange-Lie-algebroid}
\displaystyle \frac{dx^{i}}{dt} = \rho^{i}_{\alpha}y^{\alpha}, \; \; \; \frac{d}{dt}\left(\frac{\partial L}{\partial y^{\alpha}}\right) = \rho^{i}_{\alpha} \frac{\partial L}{\partial x^{i}} - C_{\alpha \beta}^{\gamma}y^{\beta}\frac{\partial L}{\partial y^{\gamma}},
\end{equation}
(for more details, see \cite{LeMaMa,Ma0}).

\medskip

\noindent {\bf Some particular cases:} \cite{CoLeMaMaMa,LeMaMa,Ma0,We}

i) If $A$ is the standard Lie algebroid $\tau_{TQ}: TQ \to Q$ then Eqs (\ref{Euler-Lagrange-Lie-algebroid}) are just the Hamel equations (the standard Euler-Lagrange equations in quasi-velocities) for an standard Lagrangian function $L: TQ \to \mathbb{R}$.

ii) If $A$ is an involutive vector subbundle of $TQ$ (for instance, the vertical bundle of a fibration) then a real function $L$ on $A$ induces a Lagrangian system subjected to holonomic constraints and (\ref{Euler-Lagrange-Lie-algebroid}) are just the holonomic equations for $L$.  

iii) If $A$ is a Lie algebra ${\frak g}$ (as a Lie algebroid over a single point) then Eqs (\ref{Euler-Lagrange-Lie-algebroid}) are the Euler-Poincar\'e equations for a Lagrangian function on the Lie algebra ${\frak g}$.

iv) If $A$ is the Atiyah algebroid $TQ/G$ associated with a free and proper action of a Lie group $G$ on $Q$ then Eqs (\ref{Euler-Lagrange-Lie-algebroid}) are the Lagrange-Poincar\'e equations for a $G$-invariant Lagrangian function. 

\subsection{Discrete Lagrangian Mechanics on Lie groupoids}\label{discrete-mechanics}

In this section,we will present a brief description of discrete Lagrangian Mechanics on Lie groupoids (for more details, see \cite{MaMaMa,We}).

\medskip

\noindent {\bf Discrete Euler-Lagrange equations}

Let $G$ be  a Lie groupoid with structural maps
\[
\alpha, \beta: G \to M, \; \; \varepsilon: M \to G, \; \; i: G \to G,
\; \; m: G_{2} \to G.
\]
Denote by $\tau:AG\to M$ the Lie algebroid of $G$ (for the definition of a Lie groupoid and its associated  Lie algebroid, see Appendix \ref{algebroide-grupoide}).

 A discrete Lagrangian is a function $\map{L_d}{G}{\Real}$. Fixed $g\in
G$, we define the set of admissible sequences with values in $G$:
\[
\begin{array}{rcl}
{\mathcal C}^N_{g}=\{(g_1, \ldots, g_N)\in G^N\; \mid \; (g_k,
g_{k+1})\in G_2 \hbox{ for all } k=1,\ldots, N-1 \\ \hbox{ and } g_1
\ldots g_N=g  \}.
\end{array}
\]
An admissible sequence $(g_{1}, \dots , g_{N}) \in {\mathcal
C}^N_g$  is a solution of the discrete Euler-Lagrange
equations if
\[
0=\sum_{k=1}^{N-1}\left[\lvec{X}_k\big({g_k})(L_d)-\rvec{X}_k\big({g_{k+1}})(L_d)
\right], \; \; \mbox{ for } X_{1}, \dots , X_{N-1} \in
\Gamma(AG).
\]
For $N=2$ we obtain that $(g, h)\in G_2$ is a solution if
\begin{equation}\label{Eqs-Euler-Lag-discretas}
\lvec{X}({g})(L_d)-\rvec{X}({h})(L_d)=0
\end{equation}
for every section $X$ of $AG$. Here, $\lvec{X}$ (resp., $\rvec{X}$) is the left-invariant (resp., right-invariant) vector field on $G$ induced by $X$ (see Appendix \ref{algebroide-grupoide}). 

In the particular case when $G \rightrightarrows M$ is the pair groupoid, the Lie groupoid associated with a fibration, a Lie group or the Atiyah groupoid associated with a free and proper action (see Appendix \ref{algebroide-grupoide}), then Eqs. (\ref{Eqs-Euler-Lag-discretas}) are just the discrete versions of the dynamical equations which we have mentioned at the end of Section \ref{Con-Lag-Mech-Lie-alg}.

\medskip

\noindent {\bf Discrete Lagrangian evolution operator}

A smooth map  $F_{L_{d}}: G\longrightarrow G$ is said to be 
a discrete flow or a discrete Lagrangian evolution
operator for a discrete Lagrangian function $L_{d}: G \to \mathbb{R}$ if it satisfies the following properties:
\begin{enumerate}
\item[-] $\hbox{graph}(F_{L_{d}})\subseteq G_2$, that is, $(g, F_{L_{d}}(g))\in G_2$, $\forall g\in
G$.
\item[-] $(g, F_{L_{d}}(g))$ is a solution of the discrete Euler-Lagrange
equations, for all $g\in G$, that is,
\begin{equation}\label{5.22'}
\lvec{X}(g)(L_d)-\rvec{X}(F_{L_{d}}(g))(L_d)=0
\end{equation}
for every section $X$ of $AG$ and every $g\in G.$
\end{enumerate}

\medskip

\noindent{\bf Discrete Legendre transformations}
\label{section5.6} 

Given a discrete La\-gran\-gian
$\map{L_{d}}{G}{\Real}$ we define two discrete Legendre
transformations $\F^{-}L_{d}: G\longrightarrow A^*G$ and
$\F^{+}L_{d}: G\longrightarrow A^*G$
 as follows 
\begin{equation}\label{DLt-}
(\F^{-}L_{d})(h)(v_{\varepsilon(\alpha(h))})=-v_{\varepsilon(\alpha(h))}(L_d\circ
r_h\circ i), \mbox{ for } v_{\varepsilon(\alpha(h))}\in
A_{\alpha(h)}G,
\end{equation}
\begin{equation}\label{DLt+}
(\F^{+}L_d)(g)(v_{\varepsilon(\beta(g))})=
v_{\varepsilon(\beta(g))}(L_d\circ l_g), \mbox{ for }
v_{\varepsilon(\beta(g))}\in A_{\beta(g)}G.
\end{equation}
Here, $l_g: \alpha^{-1}(\beta(g)) \to \alpha^{-1}(\alpha(g))$ (resp., $r_h: \beta^{-1}(\alpha(h)) \to \beta^{-1}(\beta(h))$ is the left-translation by $g \in G$ (resp., the right-translation by $h$) (see Appendix \ref{algebroide-grupoide}).  
\begin{remark}\label{r4.4'}
Note that  $(\F^{+}L_d)(g)\in A^*_{\beta(g)}G$ and $(\F^{-}L_d)(h)\in
A^*_{\alpha(h)}G$. Furthermore, if $\{X_\gamma\}$ (respectively,
$\{Y_{\mu}\}$) is a local basis of $\Gamma(AG)$ in an open subset $U$ (resp., $V$) such that
$\alpha(h) \in U$ (respectively, $\beta(g) \in V$) and
$\{X^\gamma\}$ (respectively, $\{Y^\mu\}$) is the dual basis of
$\Gamma(A^*G),$ it follows that
\[
\F^-L_{d}(h)=\rvec{X}_\gamma(h)(L_d)X^\gamma(\alpha(h)),\;\;\;
\F^+L_{d}(g)=\lvec{Y}_\mu(g)(L_d)Y^\mu(\beta(g)).
\]
\end{remark}

\medskip

\noindent{\bf Discrete regular Lagrangians}

A Lagrangian $L_d: G\to \Real$ on a Lie groupoid $G$ is said to be
regular if the Legendre transformation $\F^-L_d$ is a local
diffeomorphism (or, equivalently, the Legendre transformation $\F^+L_d$ is a local
diffeomorphism).

One may prove that the Lagrangian $L_{d}$ is
regular if and only if for every $g\in G$ and every local basis
$\{X_\gamma\}$ (respectively, $\{Y_\mu\}$) of $\Gamma(AG)$ on an
open subset $U$ (respectively, $V$) of $M$ such that $\alpha(g)\in
U$ (respectively, $\beta(g)\in V$) we have that the matrix
$\rvec{X_\gamma}(\lvec{Y_\mu}(L_d))$ is regular on
$\alpha^{-1}(U)\cap \beta^{-1}(V)$.

Moreover, if $L_d:G\to \Real$ is regular and $(g_0,h_0)\in G_{2}$
is a solution of the discrete Euler-Lagrange equations for $L_d$
then there exist two open subsets $U_0$ and $V_0$ of $G$, with
$g_0\in U_0$ and $h_0\in V_0,$ and there exists a (local) discrete
Lagrangian evolution operator $F_{L_{d}}:U_0\to V_0$ such that:
\begin{itemize}
\item $F_{L_{d}}(g_0)=h_0,$

\item $F_{L_{d}}$ is a diffeomorphism and

\item $F_{L_{d}}$ is unique, that is, if $U_0'$ is an open subset of
$G$, with $g_0\in U_0'$ and $F_{L_{d}}':U'_0\to G$ is a (local)
discrete Lagrangian evolution operator then $
F_{L_{d}}'|_{U_0\cap U_0'}=F_{L_{d}}|_{U_0\cap U_0'}$.
\end{itemize}
In fact, the discrete flow $F_{L_{d}}$ is given by
\begin{equation}\label{dis-Lag-evol-operator}
F_{L_{d}} = (\F^{-}L_d)^{-1} \circ \F^{+}L_d,
\end{equation}
where the composition of the maps $(\F^{-}L_d)^{-1}$ and $\F^{+}L_d$ makes sense (for more details, see \cite{MaMaMa}).

\medskip

\noindent {\bf Discrete Hamiltonian evolution operator}

Let $L_d: G\to \R$ be a regular discrete Lagrangian and assume, without 
loss of generality, that the Legendre transformations $\F^+L$ and
$\F^-L$ are global diffeomorphisms. Then, one may introduce the
discrete Hamiltonian evolution operator $\tilde{F}_{L_d}: A^*G\to A^*G$, by
\begin{equation}\label{dheo}
\tilde{F}_{L_{d}}=\F^+L_d\circ F_{L_{d}} \circ (\F^+L_d)^{-1},
\end{equation}
where $F_{L_{d}}$ is the discrete Lagrangian evolution operator.

Note that, from (\ref{dis-Lag-evol-operator}), we have  the following alternative
definition
\[
\tilde{F}_{L_{d}}=\F^+L_d\circ (\F^-L_d)^{-1}
\]
of the discrete Hamiltonian evolution operator. 

Moreover, one may prove that $\tilde{F}_{L_{d}}$ is a Poisson isomorphism for the canonical
Poisson bracket on $A^*G$ \cite{MaMaMa}.

\section{Convexity theorems for standard second order differential
equations}\label{section2}

Let $Q$ be a smooth manifold and $TQ$  its tangent bundle. We
will denote by $\tau_{TQ}: TQ \to Q$ the canonical projection.

A standard \sode\ on $Q$ is a \sode\ on the standard Lie algebroid $\tau_{TQ}: TQ \to Q$. The solutions of a \sode\ on $TQ$ are the solutions of a system of explicit second order differential equations. That is, if we take  canonical coordinates $(x^i, \dot{x}^i)$ on $TQ$ then a \sode\ $\Gamma$   is locally written as follows: 
\[
\Gamma=\dot{x}^i\frac{\partial}{\partial x^i}+\Gamma^i(x, \dot{x})\frac{\partial}{\partial \dot{x}^i}\; ,
\]
and its solutions are the curves $t\rightarrow (x^i(t))$ such that  
\[
\frac{d^2 x^i}{dt^2}(t)=\Gamma^i(x(t), \frac{dx}{dt}(t))\; .
\] 
Therefore, a first convexity
theorem for a \sode\   may be deduced using the theory of explicit second order differential equations. 


%
%
\begin{theorem}\label{convexity1}
Let $\Gamma$ be a \sode\ in $Q$ and $q_{0}$ be a point of $Q$. Then,
one may find a sufficiently small positive number $h_{0}$, and two
open subsets $U$ and $\tilde{U}$ of $Q$, with $q_{0} \in U
\subseteq \tilde{U}$, such that for all $q_{1} \in U$ there exists
a unique trajectory of $\Gamma$
\[
\sigma_{q_0q_1}: [0, h_{0}] \to \tilde{U} \subseteq Q
\]
satisfying
\[
\sigma_{q_0q_1}(0) = q_0, \; \; \; \sigma_{q_0q_1}(h_0) = q_1.
\]
\end{theorem}
\begin{proof} A proof of this result may be found in Appendix \ref{Hartmann}.
\end{proof}

\begin{remark}{\rm Note that if we have a positive number $h_{0} >
0$ as in Theorem \ref{convexity1} then for every $h$, $0< h \leq
h_{0}$, there exist two open subsets $U_{h}$ and $\tilde{U}_{h}$
of $Q$, with $q_{0} \in U_{h} \subseteq \tilde{U}_{h}$, such that
for all $q_{1} \in U_{h}$ there is a unique trajectory of
$\Gamma$, $\sigma_{q_0q_1}: [0, h] \to \tilde{U}_{h} \subseteq Q$,
satisfying
\[
\sigma_{q_0q_1}(0) = q_0, \; \; \; \sigma_{q_0q_1}(h) = q_{1}
\]
(see the proof of Theorem \ref{convexity1} in Appendix \ref{Hartmann}).}
\end{remark}
Let $\Gamma$ be a \sode\ on $Q$.

We will denote by $\Phi^{\Gamma}$ the flow of $\Gamma$
\[
\Phi^{\Gamma}: D^{\Gamma}\subseteq \mathbb{R} \times TQ \to TQ.
\]
Here, $D^{\Gamma}$ is the open subset of $\mathbb{R} \times TQ$
given by
\[
D^{\Gamma} = \{(t, v) \in \mathbb{R} \times TQ \mid \Phi^{\Gamma}(\cdot, v) \mbox
{ is defined at least in } [0, t] \}.
\]
Now, if $q_{0}$ is a point of $Q$ and $h_{0} \geq 0$, we may
consider the open subset $D^{\Gamma}_{(h_{0}, q_{0})}$ of
$T_{q_{0}}Q$ given by
\[
D^{\Gamma}_{(h_{0},q_{0})} = \{v \in T_{q_{0}}Q \mid (h_{0}, v) \in
D^{\Gamma} \}.
\]
Note that if $h_{0} > 0$ is sufficiently small then it is clear
that $D^{\Gamma}_{(h_{0},q_{0})} \neq \emptyset $. Moreover, we
may introduce the exponential map associated with $\Gamma$
at $q_{0}$ for the time $h_{0}$ as follows
\[
exp^{\Gamma}_{(h_{0}, q_{0})}(v) = (\tau_{Q} \circ
\Phi^{\Gamma})(h_{0}, v), \; \; \mbox{ for } v\in
D^{\Gamma}_{(h_{0}, q_{0})}.
\]
We remark that the map $exp^{\Gamma}_{(0, q_{0})}$ is constant.
However, if we choose the positive number $h_{0}$ as in Theorem
\ref{convexity1}, we have that the map
\[
exp^{\Gamma}_{(h_{0}, q_{0})}: {\mathcal U} \subseteq
D^{\Gamma}_{(h_{0}, q_{0})} \to U \subseteq Q
\]
is a diffeomorphism, where ${\mathcal U}$ is the open subset of
$D^{\Gamma}_{(h_{0}, q_{0})} \subseteq T_{q_{0}}Q$ defined by
\[
{\mathcal U} = \{\dot{\sigma}_{q_0q_1}(0)\in T_{q_0}Q \mid q_{1} \in
U \}.
\]
In other words, the map $exp^{\Gamma}_{(h_{0}, q_{0})}:
D^{\Gamma}_{(h_{0},q_{0})}\subseteq T_{q_0}Q \to Q$ is
non-singular at the point $\dot{\sigma}_{q_0q_0}(0) = v_{q_{0}}
\in D^{\Gamma}_{(h_{0}, q_{0})}$.

Next, if $h_{0} \geq 0$ is sufficiently small, we will consider
the open subset $D^{\Gamma}_{h_{0}}$ of $TQ$ given by
\[
D^{\Gamma}_{h_{0}} = \{v \in TQ \mid (h_{0}, v) \in D^{\Gamma} \}.
\]
Note that
\[
v \in D^{\Gamma}_{h_{0}} \Longrightarrow D^{\Gamma}_{(h_{0},
\tau_{Q}(v))} = D^{\Gamma}_{h_{0}} \cap T_{\tau_{Q}(v)}Q \subseteq
D^{\Gamma}_{h_{0}}.
\]
Thus, since $\tau_Q: TQ \to Q$ is an open map, it follows that $\tau_Q(D^{\Gamma}_{h_0})$ is an open subset 
of $Q$ and
\[
D^{\Gamma}_{h_0} = \bigcup_{q\in \tau_Q(D^{\Gamma}_{h_0})} D^{\Gamma}_{(h_0, q)}.
\]
In addition, we may define the smooth map $exp_{h_{0}}^{\Gamma}:
D^{\Gamma}_{h_{0}} \subseteq TQ \to Q \times Q$ as follows
\[
exp^{\Gamma}_{h_{0}}(v) = (\tau_{Q}(v), exp^{\Gamma}_{(h_{0},
\tau_{Q}(v))}(v)), \; \; \mbox{ for } v \in D^{\Gamma}_{h_{0}}.
\]
Moreover, we deduce that
\begin{lemma}\label{non-singular}
Let $v$ be an element of $D^{\Gamma}_{h_{0}}$ such that
$exp^{\Gamma}_{(h_{0}, \tau_{Q}(v))}$ is non-singular at $v$.
Then, $exp^{\Gamma}_{h_{0}}$  is also non-singular at $v$.
\end{lemma}
\begin{proof}
We must prove that the map
\[
T_{v}(exp^{\Gamma}_{h_{0}}): T_v(D^{\Gamma}_{h_{0}}) \simeq
T_v(TQ) \to T_{(\tau_{Q}(v), exp^{\Gamma}_{(h_{0},
\tau_{Q}(v))}(v))}(Q\times Q) \simeq T_{\tau_{Q}(v)}Q \times
T_{exp^{\Gamma}_{(h_{0},\tau_{Q}(v))}(v)}Q
\]
is a linear isomorphism.

Suppose that
\[
0 = (T_{v}(exp_{h_{0}}^{\Gamma}))(X_{v}), \; \; \mbox{ with }
X_{v} \in T_{v}(D^{\Gamma}_{h_{0}}).
\]
Then, we have that
\[
0 = (T_{v}\tau_{Q})(X_{v}) \; \; \mbox{ and } \; \; 0 =
(T_{v}exp^{\Gamma}_{(h_{0},\tau_{Q}(v))})(X_{v}).
\]
The first condition implies that
\[
X_{v} \in T_{v}(D^{\Gamma}_{h_{0}} \cap T_{\tau_{Q}(v)}Q) =
T_{v}(D^{\Gamma}_{(h_{0}, \tau_{Q}(v))})
\]
and thus, using the second one, we conclude that
\[
X_{v} = 0.
\]
\end{proof}
As we know, if $h_{0} > 0$ is sufficiently small and $q_{0} \in Q$
then the map $exp_{(h_{0}, q_{0})}^{\Gamma}: D^{\Gamma}_{(h_{0},
q_{0})} \to Q$ is non-singular at the point
$\dot{\sigma}_{q_0q_0}(0) = v_{q_0} \in D^{\Gamma}_{(h_{0},
q_{0})}$. Therefore, using Lemma \ref{non-singular}, we deduce the following result

\begin{theorem}\label{convexity2}
Let $\Gamma$ be a \sode\ in $Q$ and $q_{0}$ be a point of $Q$. Then,
one may find a sufficiently small positive number $h_{0}$, an open
subset ${\mathcal U} \subseteq D^{\Gamma}_{h_0} \subseteq TQ$, with $\dot{\sigma}_{q_0q_0}(0)
\in {\mathcal U}$, and open subsets $U, \tilde{U}$ of $Q$, with $q_{0} \in
U\subseteq \tilde{U}$, such that:
\begin{enumerate}
\item
The map
\[
exp^{\Gamma}_{h_{0}}: {\mathcal U} \subseteq D^{\Gamma}_{h_{0}} \to U \times U \subseteq Q \times Q
\]
is a diffeomorphism.
\item
 For every couple $(q, q') \in U \times U$ there
exists a unique trajectory of $\Gamma$
\[
\sigma_{qq'}: [0, h_{0}] \to Q
\]
in $\tilde{U}$ satisfying
\[
\sigma_{qq'}(0) = q, \; \; \sigma_{qq'}(h_{0}) = q' \; \; \mbox{
and } \; \; \dot{\sigma}_{qq'}(0) \in {\mathcal U}.
\]
\end{enumerate}
\end{theorem}

We will denote by 
$R^{e^-}_{h_{0}}: U \times U \to {\mathcal U}$ (respectively, $R^{e+}_{h_{0}}: U \times U \to {\mathcal U}$)
 the inverse map of the diffeomorphism 
$exp^{\Gamma}_{h_{0}}: {\mathcal U}  \to U \times U $ (respectively, $exp^{\Gamma}_{h_{0}} \circ \Phi^{\Gamma}_{-h_{0}}: 
\Phi^{\Gamma}_{h_{0}}({\mathcal U}) \to U \times U$).

The maps 
\[ 
R^{e^-}_{h_{0}}: U \times U \subseteq Q \times Q \to {\mathcal U}\subseteq TQ \mbox{ and }
 R^{e^+}_{h_{0}}: U \times U \subseteq Q \times Q \to \Phi^{\Gamma}_{h_0}({\mathcal U})\subseteq TQ
 \]
are called the exact retraction maps associated with $\Gamma$. We have that
\[
R^{e^-}_{h_{0}}(q, q') = \dot{\sigma}_{qq'}(0), \; \; R^{e^+}_{h_{0}}(q, q') = \dot{\sigma}_{qq'}(h_0).
\]
Note that
\[
R^{e^+}_{h_{0}} = \Phi^{\Gamma}_{h_0} \circ R^{e^-}_{h_{0}}.
\]

\[
\xymatrix{ {\mathcal U}\subseteq TQ
\ar[dd]_{\Phi_{h_0}^\Gamma} &   & &U\times U\subseteq Q\times Q \ar[lll]_{R^{e^-}_{h_0}}\ar[ddlll]^{{R^{e^+}_{h_0}}}\\
  &  & &\\
\Phi_{h_0}^\Gamma( {\mathcal U})\subseteq TQ&  & & }
\]



\let\s\alpha
\let\t\beta
\let\e\epsilon
\def\rightrightarrow{\rightrightarrows}
\def\T{\CMcal{T}}


\begin{remark}\label{reweyls}
{\rm 
In some applications, it is useful to define the following map
\begin{equation}\label{weyl}
\widetilde{exp}^{\Gamma}_{h_0}(v)=(exp^{\Gamma}_{(-h_0/2, \tau_Q(v))}(v), exp^{\Gamma}_{(h_0/2, \tau_Q(v))}(v))
\end{equation}
which is a local diffeomorphism since
\[
\widetilde{exp}^{\Gamma}_{h_0}=exp_{h_0}^{\Gamma}\circ \Phi^{\Gamma}_{-h_0/2}\; .
\]
We recall that $\Phi^{\Gamma}$ is the flow of $\Gamma$.
}
\end{remark}

\section{Convexity theorems for second order differential equations on Lie algebroids}\label{sec4}

In this section, we will obtain a version of Theorem \ref{convexity2} for a \sode\  on a general Lie algebroid
$A$.

First of all, in Section \ref{fibration}, we will discuss the particular case when $A$ is a special
integrable Lie algebroid: the vertical bundle associated with a fibration. In such a case, we will
prove a parametrized version of Theorem \ref{convexity2}. Next, in Section \ref{integrable}, we will consider
the more general case when $A$ is the Lie algebroid $AG$ associated with an arbitrary Lie
groupoid $G$. In fact, we will see that a \sode\ on $AG$ induces a \sode\ on the vertical
bundle of the source map of $G$ and, then, we will apply the results of the previous section. Finally, we will 
discuss the general case. For this purpose, we will use that for every Lie algebroid $A$ there exists a
(local) Lie groupoid whose Lie algebroid is $A$.

\subsection{The particular case of the vertical bundle of a fibration}\label{fibration}
 
Consider a surjective submersion $\map{\pi}{P}{M}$ and the pair
groupoid $P\times P\rightrightarrow P$ (see Appendix \ref{algebroide-grupoide}) . The subset $G\pi\subset P\times P$
given by
\[
G\pi=\set{(p_1,p_2)\in P\times P}{\pi(p_1)=\pi(p_2)}
\]
is a Lie subgroupoid of the pair groupoid. In consequence, the source and target
maps are $(p_1,p_2)\mapsto p_1$ and $(p_1,p_2)\mapsto p_2$,
respectively, the identity map is $p\mapsto (p,p)$ and the
multiplication is given by $(p_1,p_2)(p_2,p_3)=(p_1,p_3)$. The Lie
algebroid of $G\pi$ is the vector bundle whose fiber at a point
$p$ is
\[
A_p(G\pi)=\set{(0,v)\in T_pP\times T_pP}{T\pi(v)=0}
\]
with the anchor $(0,v)\mapsto v$, and hence it can be identified
with the vertical bundle $\map{\tau}{V\pi=\ker{T\pi}}{P}$ with the
canonical inclusion as the anchor map and Lie bracket in the space 
of sections $\Gamma(\tau)$ the restriction to $\Gamma(V\pi)$ of 
the standard Lie bracket of vector fields.

On $V\pi$ we can take coordinates as follows. We consider local
coordinates $(x^i)=(x^a,x^\alpha)$ in $P$ adapted to the submersion
$\pi$, that is, $\pi(x^a,x^\alpha)=(x^a)$. The coordinate vector
fields $\{e_\alpha=\partial/\partial x^\alpha\}$ are a basis of
local sections of $V\pi$, and hence we have coordinates
$(x^a,x^\alpha,y^\alpha)$ on $V\pi$, where $y^\alpha$ are the
components of a vector on such a coordinate basis. In these
coordinates the structure functions are
\[
\rho^a_\alpha=0\qquad
\rho^\beta_\alpha=\delta_\alpha^\beta\qquad\text{and}\qquad
C^\alpha_{\beta\gamma}=0.
\]

A \sode\ vector field on $V\pi$ is a vector field
$\Gamma\in\vectorfields{V\pi}$ such that
$T_a\tau(\Gamma(a))=\rho(a)$ for every $a\in V\pi$. In other words
$T_a\tau(\Gamma(a))$ is the vertical vector $a\in V\pi$ itself, and
hence, if $m=\pi(\tau(a))$ then the vector $\Gamma(a)$ is tangent
to $T(\pi^{-1}(m))$. It follows that a \sode\ on $V\pi$ is but a
parametrized version of an ordinary \sode, where the parameters
are the coordinates on the base manifold $M$. In other words,  it is
a smooth family of ordinary \sode s, one on each fiber of the
projection $\pi: P \to M$.

This can also be easily seen in coordinates. Locally,  such a
\sode\ vector field is
\begin{align*}
\Gamma&=\rho^a_\alpha y^\alpha\pd{}{x^a}+\rho^\beta_\alpha y^\alpha\pd{}{x^\alpha}+f^\alpha(x^b,x^\beta,y^\beta)\pd{}{y^\alpha}\\
&=y^\alpha\pd{}{x^\alpha}+f^\alpha(x^b,x^\beta,y^\beta)\pd{}{y^\alpha}
\end{align*}
for some local functions $f^\alpha\in\cinfty{V\pi}$. The integral
curves of $\Gamma$ are the solutions of the differential equations
\[
\dot{x}^a=0,\qquad \dot{x}^\alpha=y^\alpha,\qquad
\dot{y}^\alpha=f^\alpha(x^b,x^\beta,y^\beta)
\]
or in other words
\begin{align*}
\dot{x}^a&=0\\
\ddot{x}^\alpha&=f^\alpha(x^b,x^\beta,\dot{x}^\beta).
\end{align*}
From this expression it is obvious that a \sode\ on $V\pi$ is
locally a parametrized version of an ordinary \sode, where the
parameters are the coordinates $(x^a)$ on the base manifold $M$.

On the other hand, $V\pi$ is a regular submanifold of $TP$. In
fact, the canonical inclusion $i_{V\pi}: V\pi \to TP$ is an
embedding. In addition, if $p_{0}$ is a point of $P$ then it is
easy to prove that:
\begin{enumerate}
\item
There exists an open subset $W$ of $P$, with $p_{0} \in W$, and
there exists an standard local \sode\ $\bar{\Gamma}$ on $P$, which
is defined in $TW$, such that
\[
\bar{\Gamma}_{|TW\cap V\pi} = \Gamma_{|TW\cap V\pi}
\]
\item
If $\bar{\sigma}: [0, h] \to W\subseteq P$ is a trajectory of
$\bar{\Gamma}$ and $\pi(\bar{\sigma}(0)) = \pi(\bar{\sigma}(h)) =
m$ then $\bar{\sigma}([0, h]) \subseteq \pi^{-1}(m)$ and
$\bar{\sigma}$ is a trajectory of the standard \sode\
$\Gamma_{|T(\pi^{-1}(m))}$.
\end{enumerate}
Note that the local equations  defining  $V\pi$ as a
submanifold of $TP$ are $y^a = 0$ and, thus, it is
sufficient to take
\[
\bar{\Gamma} = \displaystyle y^{a} \frac{\partial}{\partial
x^a} + y^{\alpha} \frac{\partial}{\partial x^{\alpha}} +
f^{\alpha}(x^b, x^{\beta}, y^{\beta})\frac{\partial}{\partial
y^{\alpha}}.
\]
Therefore, using Theorem \ref{convexity1}, one may find an open
neighborhood of $p_0$ in $P$ and a unique curve
$\bar{\sigma}_{p_0p_0} \equiv \sigma_{p_0p_0}$ on it which
connects the point $p_0$ with itself and such that it is
trajectory of $\bar{\Gamma}$, for $h$ enough small. Then, from the second condition, it
follows that the curve $\sigma_{p_0p_0}$ is contained in
$\pi^{-1}(\pi(p_{0}))$ and $v_{p_0} = \dot{\sigma}_{p_0p_0}(0) \in
V_{p_0}(\pi)$. Moreover, if we apply Theorem \ref{convexity2} to
the standard \sode\ $\bar{\Gamma}$, we deduce the following result
\begin{theorem}
\label{caso.Vpi} Let $\Gamma$ be a \sode\ on the Lie algebroid
$V\pi\to P$ and let $p_0$ be a point in $P$. Then, there exists a
sufficiently small positive number $h_{0} > 0$,  an open subset
${\mathcal U} \subseteq V\pi$, with $v_{p_{0}} \in {\mathcal U}$,
and open subsets $U, \tilde{U}$ of $P$, with $p_{0} \in U\subseteq \tilde{U}$, such that 
\begin{enumerate}
\item
The exponential map of $\Gamma$ at $h_0$
\[
exp^{\Gamma}_{h_{0}}: {\mathcal U} \to (U \times U)\cap G\pi, \; \;  v \in {\mathcal U} \to (\tau_{V\pi}(v), \tau_{V\pi}(\Phi^{\Gamma}_{h_{0}}(v))),
\]
is a diffeomorphism. Here $\tau_{V\pi}: V\pi \to P$ is the canonical projection and $\Phi^{\Gamma}$ is the flow of the vector field $\Gamma$ on $V\pi$. 
\item
For
every couple $(p, p') \in (U \times U) \cap G{\pi}$ there exists
a unique trajectory $\sigma_{pp'}: [0, h_{0}] \to \pi^{-1}(\pi(p))$
in $\tilde{U} \cap \pi^{-1}(\pi(p))$ of the \sode\  $\Gamma_{|T(\pi^{-1}(\pi(p)))}$ which satisfies
\[
\sigma_{pp'}(0) = p, \; \; \sigma_{pp'}(h_{0}) = p' \; \; \mbox{
and } \dot{\sigma}_{pp'}(0) \in {\mathcal U}.
\]
\end{enumerate}
\end{theorem}

As in the standard case, we will denote by
\[
R^{e^-}_{h_{0}}: (U \times U)\cap G\pi \subseteq P \times P \to {\mathcal U}\subseteq V\pi  \mbox{ and }
 R^{e^+}_{h_{0}}: (U \times U)\cap G\pi \subseteq P \times P \to \Phi^{\Gamma}_{h_0}({\mathcal U})\subseteq V\pi
 \]
the exact retraction maps associated with $\Gamma$. In other words,
\[
R^{e^-}_{h_{0}} = (exp^{\Gamma}_{h_0})^{-1} \; \; \mbox{ and }  \; \; R^{e^+}_{h_{0}} = \Phi_{h_0}^{\Gamma} \circ R^{e-}_{h_0}.
\]

\subsection{The particular case of the Lie algebroid of a Lie groupoid}\label{integrable}

We will consider a Lie groupoid 
$G\rightrightarrow M$ with source map
$\alpha$, target map $\beta$, and consider the fibration
$\pi\equiv\alpha$ and the associated Lie algebroid $V\alpha$ as
above. Let $\map{\tau}{AG}{M}$ be the Lie algebroid of $G$, and denote
by $\rho$ its anchor.  Denote by $\Psi$ the vector bundle map
$\map{\Psi}{V\alpha}{AG}$ given by $\Psi(v_g)=Tl_{g^{-1}}v_g$,
for every $v_g\in V\alpha$. This map is well defined since
$T\alpha(Tl_{g^{-1}}v_g)=T\alpha(v_g)=0$ and hence
$Tl_{g^{-1}}v_g$ is $\alpha$-vertical at the identity in
$\beta(g)$. The following commutative diagram illustrates the situation:
\[
\xymatrix{ V\alpha
\ar[dd]_{\tau_{V\alpha}} \ar[rr]^{\Psi}&    &AG \ar[dd]^{\tau}\\
  &  & \\
G\ar[rr]^{\beta}&  & M }
\]

Moreover, $\Psi: V\alpha \to AG$ is a Lie algebroid morphism. This follows using that
if $X \in \Gamma(AG)$ and $\lvec{X}$ is the corresponding left invariant vector field on  
$G$, then $\lvec{X}$ is a section of the vector bundle $\tau_{V\alpha}: V\alpha \to G$ and:
\begin{enumerate}
\item
$\lvec{X}$ and $X$ are $\Psi$-related;
\item
If $\lcf \cdot, \cdot \rcf $ is the Lie bracket in $\Gamma(AG)$
\[
\lvec{\lcf X, Y \rcf} = [\lvec{X}, \lvec{Y}], \; \; \; \mbox{ for } X, Y \in \Gamma(AG)
\]
and
\item
The vector field $\lvec{X}$ on $G$ is $\beta$-projectable over $\rho(X)$, where $\rho$ is the anchor map
of the Lie algebroid $AG$.
\end{enumerate}

On the other hand, given a \sode\ $\Gamma$ in $AG$ there exists a unique \sode\
$\tilde{\Gamma}$ in $V\alpha$ which is $\Psi$ related to $\Gamma$,
that is $T\Psi\circ\tilde{\Gamma}=\Gamma\circ\Psi$. Indeed, this
is a special case of the following result, by taking into account
that $\Psi$ is a fiberwise bijective morphism of Lie algebroids.

\begin{proposition}\label{rel-SODES}
Let $\map{\tau_1}{E_1}{M_1}$ and $\map{\tau_2}{E_2}{M_2}$ be Lie
algebroids and let $\map{\Psi}{E_1}{E_2}$ be a morphism of Lie
algebroids which is fiberwise bijective. Given a \sode\ vector
field $\Gamma_2$ on the Lie algebroid $E_2$ there exists a unique
\sode\ vector field $\Gamma_1$ on the Lie algebroid $E_1$ such
that $T\Psi\circ\Gamma_1=\Gamma_2\circ\Psi$.
\end{proposition}
\begin{proof}
We have to show that for every $a_1\in E_1$ there exists a unique
$v_1\in T_{a_1}E_1$ satisfying the equations
\[
T\Psi(v_1)=\Gamma_2(\Psi(a_1)),\qquad\text{and}\qquad
T\tau_1(v_1)=\rho_1(a_1).
\]
Note that if $v_{1}, v_{1}' \in T_{a_{1}}E_{1}$ satisfy these
conditions then $v_{1}' - v_{1} \in Ker (T\tau_{1})$ and, since
$\Psi$ is fiberwise bijective and $(T\Psi)(v_1) =
(T\Psi)(v_{1}')$, we conclude that $v_{1}' = v_{1}$.

Next, we will see that one may find a vector $v_{1} \in
T_{a_{1}}E_{1}$ which satisfies the above equations.

 For that consider a fixed (but
arbitrary) auxiliary \sode\ vector field
$\Gamma\in\vectorfields{E_1}$. Since $\Gamma(a_1)$ projects to
$\rho_1(a_1)$, the vector $v_1$ satisfies the second equation if
and only if the vector $w_1=v_1-\Gamma(a_1)$ is vertical. If $\xi^V: E_1\times E_1\rightarrow V\tau_1$ is the canonical vertical lift, it follows that  we can write  $\xi^V(a_1,c_1)=w_1=v_1-\Gamma(a_1)$ for
a unique $c_1\in E_1$, and then, using that $\Psi$ is a morphism
of Lie algebroids, the first equation reads
\[
\xi^V(\Psi(a_1),\Psi(c_1))=\Gamma_2(\Psi(a_1))-T\Psi(\Gamma(a_1)).
\]
The right hand side of this equation is vertical at the point
$\Psi(a_1)$, and since $\Psi$ is fiberwise bijective it has a
unique solution $c_1$. Thus, the vector
$v_1=\Gamma(a_1)+\xi^V(a_1,c_1)$ is the solution for our
equations.
\end{proof}



Let $m_0\in M$ and consider the identity $\varepsilon(m_0)\in G$.
For $\pi=\alpha$, we can now apply Theorem~\ref{caso.Vpi} to the
unique \sode\ $\tilde{\Gamma}$ in $V\alpha$ which is $\Psi$
related to $\Gamma$, and the point $p_0=\varepsilon(m_0)$.

As we know, one may find an open neighborhood of $\varepsilon(m_0)$
in $G$ and a unique curve
$\tilde{\sigma}_{\varepsilon(m_{0})\varepsilon(m_{0})}$ on it which
connects the point $\varepsilon(m_{0})$ with itself and such that it
is a trajectory of $\tilde{\Gamma}_{|T(\alpha^{-1}(m_0))}$. In fact,
the curve $\tilde{\sigma}_{\varepsilon(m_{0})\varepsilon(m_{0})}$ is
contained in $\alpha^{-1}(m_{0})$ and, therefore, $v_{m_{0}} =
\dot{\tilde{\sigma}}_{\varepsilon(m_{0})\varepsilon(m_{0})}(0) \in
V_{\varepsilon(m_{0})}\alpha = A_{m_{0}}G$.

Moreover, we may prove the following result

\begin{theorem}\label{convexity-definitivo}
Let $\Gamma$ be a \sode\ vector field on the Lie algebroid $AG\to
M$ of the Lie groupoid $G\rightrightarrow M$,  $m_0\in M$  a point in the
base manifold and $\tilde{\Gamma}$ the corresponding \sode\ in the Lie algebroid $V\alpha \to G$. 
Then, there exists a sufficiently small positive
number $h > 0$, an open subset ${\mathcal U}$ in $AG$, with
$v_{m_0} \in {\mathcal U}$, and open subsets $U, \tilde{U}$ of $G$, with
$\varepsilon(m_{0}) \in U \subseteq \tilde{U}$, such that:
\begin{enumerate}
\item
The exponential map associated with $\Gamma$ at $h$
\[
exp^{\Gamma}_h: {\mathcal U} \to U, \; \; v \in {\mathcal U} \to \tau_{V\alpha}(\Phi^{\tilde{\Gamma}}_h(v)) \in U
\]
is a diffeomorphism. Here $\tau_{V\alpha}: V\alpha \to G$ is the canonical projection and $\Phi^{\tilde{\Gamma}}$ is 
the flow of the vector field $\tilde{\Gamma}$ on $V\alpha$. 
\item
For every $g\in U$ there exists
a unique trajectory $\sigma_{\varepsilon(\alpha(g))g}: [0, h] \to
\alpha^{-1}(\alpha(g))$ of $\tilde{\Gamma}$ in $\tilde{U} \cap \alpha^{-1}(\alpha(g))$ satisfying the following conditions
\[
\sigma_{\varepsilon(\alpha(g))g}(0) = \varepsilon(\alpha(g)), \; \;
\sigma_{\varepsilon(\alpha(g))g}(h) = g, \; \;
\dot{\sigma}_{\varepsilon(\alpha(g))g}(0) \in {\mathcal U}.
\]
Thus, the induced curve $a_{\varepsilon(\alpha(g))g} = \Psi \circ
\dot{\sigma}_{\varepsilon(\alpha(g))g}$ in $AG$ is an integral curve of
the \sode\ $\Gamma$. In fact, $a_{\varepsilon(\alpha(g))g}$ is the integral curve of $\Gamma$ with
initial condition $\dot{\sigma}_{\varepsilon(\alpha(g))g}(0)$ and, moreover,
\[
a_{\varepsilon(\alpha(g))g}(h) = T_g l_{g^{-1}}(\dot{\sigma}_{\varepsilon(\alpha(g))g}(h)).
\] 
\end{enumerate}
\end{theorem}
\begin{proof}
Using Theorem \ref{caso.Vpi}, we deduce that there exists a sufficiently small positive number $h >0$, an open subset
$\tilde{\mathcal U}$ in $V\alpha$, with $v_{m_0} \in \tilde{\mathcal U}$, and open subsets $V, \tilde{V}$ of $G$, 
with $\varepsilon(m_0) \in V \subseteq \tilde{V}$, such that:
\begin{enumerate}
\item
The exponential map of $\tilde{\Gamma}$ at $h$
\[
exp_{h}^{\tilde{\Gamma}}: \tilde{\mathcal U} \subseteq V\alpha \to (V \times V) \cap G\alpha \subseteq G \times G, \; \; \tilde{v} \to (\tau_{V\alpha}(\tilde{v}), \tau_{V\alpha}(\Phi^{\tilde{\Gamma}}_{h}(\tilde{v}))),
\]
is a diffeomorphism.
\item
For every $g, g' \in (V \times V)\cap G\alpha$, there exists a unique trajectory $\sigma_{gg'}: [0, h] \to \alpha^{-1}(\alpha(g))$ of $\tilde{\Gamma}$ in $\tilde{U} \cap \alpha^{-1}(\alpha(g))$ satisfying the following conditions
\[
\sigma_{gg'}(0) = g, \; \; \; \sigma_{gg'}(h) = g', \; \; \; \dot{\sigma}_{gg'}(0) \in \tilde{\mathcal U}.
\]
\end{enumerate}
Now, we take the open subset $\tilde{\mathcal U} \cap AG$ of $AG$. It is clear that $v_{m_0} \in \tilde{\mathcal U}\cap AG$. 

Denote by $\iota: AG \to V\alpha$ the canonical inclusion. Then, the exponential map $exp_h^{\Gamma}: \tilde{\mathcal U}\cap AG \subseteq AG \to G$ is given by
\[
exp^{\Gamma}_h = pr_2 \circ exp_{h}^{\tilde{\Gamma}} \circ \iota,
\]
where $pr_2: (V\times V)\cap G\alpha \to V\subseteq G$ is the canonical projection on the second factor. In fact,
\begin{equation}\label{expresion-expo-tilde}
exp_{h}^{\tilde{\Gamma}}(\iota(v)) = (\varepsilon(\tau(v)), exp_{h}^{\Gamma}(v)).
\end{equation}
Next, we will see that the map $exp_h^{\Gamma}: \tilde{\mathcal U} \cap AG \to G$ is a local diffeomorphism.
Suppose that $X_v \in T_v(\tilde{\mathcal U} \cap AG)$, with $v \in \tilde{\mathcal U} \cap AG$, and
\[
0 = (T_v exp_h^{\Gamma})(X_v) = (T_v(\tau_{V\alpha} \circ \Phi_{h}^{\tilde{\Gamma}}))(X_v).
\]
This implies that
\[
(T_v(\alpha \circ \tau_{V\alpha} \circ \Phi_{h}^{\tilde{\Gamma}}))(X_v) = 0.
\]
But, since the trajectory of $\tilde{\Gamma}$ over a point $m$ of $\tau(\tilde{\mathcal U}\cap AG)$ is contained in the fiber
$\alpha^{-1}(m)$, we deduce that
\[
\alpha \circ \tau_{V\alpha} \circ \Phi_{h}^{\tilde{\Gamma}} \circ \iota = \tau.
\]
Thus, we have that
\[
(T_v \tau)(X_v) = 0.
\]
Therefore, from (\ref{expresion-expo-tilde}), we deduce that
\[
(T_v(exp_h^{\tilde{\Gamma}} \circ \iota)) (X_v) = 0
\]
and it follows that $X_v = 0$.

We conclude that there exists an open subset ${\mathcal U}' \subseteq AG$, with $v_{m_0} \in {\mathcal U}'$, and an open
subset $U' \subseteq G$, such that $\varepsilon (m_0) \in U' \subseteq \tilde{V}$ and
\[
exp_h^{\Gamma}: {\mathcal U}' \subseteq AG \to U' \subseteq G
\]
is a diffeomorphism.

Next, using that $\varepsilon: M \to G$ is a continuous map, we have that there exists an open subset $W$ of $M$ such that 
$m_0 \in W$ and $\varepsilon(W) \subseteq U'$. So, $\alpha^{-1}(W)$ and $U = U' \cap \alpha^{-1}(W)$ are open subsets of $G$ and
\[
\varepsilon (m_0) \in U \subseteq \tilde{V}, \; \; \varepsilon(\alpha(U))\subseteq U.
\]
Thus, we may take
\[
{\mathcal U} = (exp_h^{\Gamma})^{-1}(U) \subseteq AG, \; \; \; \; \tilde{U} = \tilde{V}
\]
and (i) and (ii) in the theorem hold.

Finally, using that the \sode\ $\tilde{\Gamma}$ is $\Psi$-related with the \sode\  $\Gamma$, we deduce the last
part of the theorem.
\end{proof}
\begin{remark}
The conditions satisfied by the curves $\sigma_{\varepsilon(\alpha(g))g}$ and $a_{\varepsilon(\alpha(g))g}$ in the previous theorem can be interpreted
in terms of $AG$-homotopy of paths (see~\cite{CrFe} for the
definitions). Indeed, if we reparametrize the curve $a_{\varepsilon(\alpha(g))g}$ and define
the curve $\bar{a}_{\varepsilon(\alpha(g))g}: [0,1] \to AG$ by $\bar{a}_{\varepsilon(\alpha(g))g}(s)=ha_g(sh)$, and
similarly we reparametrize $\sigma_{\varepsilon(\alpha(g))g}$ and define
$\bar{\sigma}_{\varepsilon(\alpha(g))g}: [0,1] \to G$ by $\bar{\sigma}_{\varepsilon(\alpha(g))g}(s)=\sigma_{\varepsilon(\alpha(g))g}(sh)$,
then the curve $\bar{a}_{\varepsilon(\alpha(g))g}$ is an $AG$-path in the $AG$-homotopy
class defined by the element   $g\in G$. Indeed, it is clear that
$\bar{a}_{\varepsilon(\alpha(g))g}(t)=Tl_{\bar{\sigma}_{\varepsilon(\alpha(g(t)))g(t)^{-1}}}(\frac{d \bar{\sigma}_{\varepsilon(\alpha(g))g}}{dt}_{|t})$, and
that $\bar{\sigma}_{\varepsilon(\alpha(g))g}(0)=\varepsilon(\alpha(g))$ and $\bar{\sigma}_{\varepsilon(\alpha(g))g}(1)=g$.
 \end{remark}
We will denote by
\[
R^{e^-}_h: U \subseteq G \to {\mathcal U}\subseteq AG \mbox{ and } R^{e^+}_h: U \subseteq G \to \Phi^{\Gamma}_h({\mathcal U})\subseteq AG
\]
the inverse maps of the diffeomorphisms $exp^{\Gamma}_h: {\mathcal U} \to U$  and $exp^{\Gamma}_h \circ \Phi^{\Gamma}_{-h}: \Phi^{\Gamma}_{h}({\mathcal U}) \to U$, respectively. They are the exact retraction maps associated with $\Gamma$ at $h$.

Note that
\begin{equation}\label{retraction-flow1}
R^{e^-}_h(g) = a_{\varepsilon(\alpha(g))g}(0) = \dot{\sigma}_{\varepsilon(\alpha(g))g}(0)
\end{equation}
and
\begin{equation}\label{retraction-flow2}
R^{e^+}_h(g) = \Phi^{\Gamma}_h(R^{e^-}_h(g)) = a_{\varepsilon(\alpha(g))g}(h) = T_g l_{g^{-1}}(\dot{\sigma}_{\varepsilon(\alpha(g))g}(h)).
\end{equation}
The following diagram illustrates the situation
\[
\xymatrix{ {\mathcal U}\subseteq AG\ar@/^/[rrr]^{exp^{\Gamma}_{h_0}}
\ar[dd]_{\Phi_{h_0}^\Gamma} &   & &U\subseteq G \ar@/^/[lll]^{R^{e^-}_{h_0}}\ar[ddlll]^{{R^{e^+}_{h_0}}}\\
  &  & &\\
\Phi^{\Gamma}_{h_0}({\mathcal U})  \subseteq AG&  & & }
\]

\subsection{The general case}

In the general case, when we have a general Lie algebroid $(E, \lcf\; ,\rcf, \rho)$, it is possible to construct a local Lie groupoid $G$ integrating this Lie algebroid. This groupoid is local in the sense that the product is not necessarily defined on $G_2$, but only locally defined near the identity section  (see \cite{CrFe} for details). In any case, Theorem \ref{convexity-definitivo} is a local result for points near of the identities and, therefore,  it remains valid for general Lie algebroids.

\section{The exact discrete Lagrangian function in the Lie groupoid setting and error analysis}\label{sec5}

\subsection{The exact discrete Lagrangian function in the Lie groupoid setting}\label{exact-discrete-Lie-groupoid}

In this section, we will introduce the exact discrete Lagrangian function associated
with a regular continuous Lagrangian function on a Lie algebroid $A$.

This construction is local and, thus, we will assume that $A$ is the Lie algebroid $AG$ associated 
with a Lie groupoid $G$ over $M$. As in the previous sections, we will denote by $\alpha$ the source
map of $G$.

Now, let $L: AG \to \R$ be a hyperregular continuous Lagrangian function on $AG$ and $\Gamma_L$ the Euler-Lagrange 
vector field on $AG$ associated with $L$.  $\Gamma_L$ is a second order differential equation 
on $AG$. Thus, if we fix a point $m \in M$ then, using Theorem \ref{convexity-definitivo}, we may find 
a sufficiently small positive number $h > 0$, an open subset $U$ in $G$, with $\varepsilon(m) \in U$ and an open subset
${\mathcal U}\subseteq AG$ such that the exponential map $exp^{\Gamma}_h: {\mathcal U} \to U$ associated with $\Gamma$ at $h$
is a diffeomorphism.

We will denote by $R^{e^-}_h: U \subseteq G \to {\mathcal U} \subseteq AG$ and $R^{e^+}_h: U \subseteq G \to \Phi^{\Gamma}_h({\mathcal U}) \subseteq AG$
the exact retraction maps associated with $\Gamma$ at $h$.

Then, we will define the exact discrete Lagrangian function $\mathbb{L}_{h}^{e}$ associated with $L$ in the open subset $U$ as follows
\begin{equation}\label{exact-lagrangian}
\mathbb{L}_h^{e}(g) = \int_0^h L(\Phi^{\Gamma_L}_t(R^{e^-}_h(g)) dt, \;\; \mbox{ for } g \in U.
\end{equation}
Using (\ref{retraction-flow1}) and the same notation as in Theorem \ref{convexity-definitivo}, we have that
\begin{equation}\label{exact-lagrangian-1}
\mathbb{L}_h^{e}(g) = \int_0^h L(a_{\varepsilon(\alpha(g))g}(t)) dt. 
\end{equation}
Since we assume that $L$ is hyperregular, we may consider
the corresponding Hamiltonian function $H: A^*G \to \R$. Moreover, the \sode\  vector field $\Gamma_L$ on $AG$ and the corresponding Hamiltonian
vector field $X_H$ on $A^*G$ are ${\mathcal F}L$-related, where ${\mathcal F}L: AG \to A^*G$ is the Legendre transformation 
associated with $L$.

Then, the aim of this section is to prove the following result.
\begin{theorem}\label{exact-discrete-hamiltonian-flow}
Let $L: AG \to \R$ be a hyperregular Lagrangian function on the Lie algebroid of a Lie groupoid $G$ over $M$. Let $m$ be a point of $M$,
$h >0$  a sufficiently small positive number, $U\subseteq G$ an open subset of $G$, with $\varepsilon(m) \in U$, and $\mathbb{L}_h^{e}: U \subseteq G \to \R$
the exact discrete Lagrangian function associated with $L$. Then:
\begin{enumerate}
\item
$\mathbb{L}_h^{e}$ is a regular discrete Lagrangian function.
\item
If $H: A^*G \to \R$ is the Hamiltonian function associated with the hyperregular Lagrangian function $L$ and $\Phi^{X_H}$ is the (local)
flow of the Hamiltonian vector field $X_H$ on $A^*G$, we have that
\[
\mathbb{F}^+\mathbb{L}^{e}_h = \Phi^{X_H}_h \circ \mathbb{F}^-\mathbb{L}^{e}_h,
\]
where $\mathbb{F}^+\mathbb{L}^{e}_h$ and $\mathbb{F}^-\mathbb{L}^{e}_h$ are the discrete Legendre transformations 
associated with $\mathbb{L}_h^{e}$.
In other words, the Hamiltonian evolution operator $\mathbb{F}^+\mathbb{L}^{e}_h \circ (\mathbb{F}^-\mathbb{L}^{e}_h)^{-1}$ 
associated with the exact discrete Lagrangian function $\mathbb{L}^{e}_h$ is just the flow of $X_H$ at time $h$ or, equivalently, the following diagram is commutative:
\end{enumerate}
\end{theorem}
\[
\xymatrix{ &   & &A^*G\ar[dd]^{\Phi_h^{X_H}}&&AG\ar[dd]^{\Phi_h^{\Gamma_L}}\ar[ll]^{{\mathcal F}L}\\
  U\subseteq G\ar[urrr]^{\mathbb{F}^-\mathbb{L}^{e}_h}\ar[drrr]^{\mathbb{F}^+\mathbb{L}^{e}_h}
  &  & & &&\\
 &  & &A^*G&&AG\ar[ll]^{{\mathcal F}L}\ }
\]

To simplify the reading, in what follows we will omit any reference to the domain of the maps that we are going to consider, and we will just indicate the spaces where they are defined. 

Le $\Psi: V\alpha \to AG$ be the  vector bundle morphism over $\beta: G \to M$ between the vector bundles $V\alpha$ and $AG$ and $\tilde{L}$ the Lagrangian function on $V\alpha$ given by
\[
\tilde{L} = L \circ \Psi.
\]
Since $\Psi$ is a fiberwise isomorphism of vector bundles, we deduce that it induces a new fiberwise isomorphism
of vector bundles $(\Psi^{-1})^*: V^{*}\alpha \to A^*G$ over $\beta: G \to M$. In fact, 
\[
((\Psi^{-1})^*(\gamma_g))(v_{\beta(g)}) = \gamma_g((\Psi_g^{-1})(v_{\beta(g)}))
\]
for $\gamma_g \in V_g^{*}\alpha$ and $v_{\beta(g)} \in A_{\beta(g)}G$. In addition, using that $\Psi$ is a Lie algebroid
morphism, it follows that $(\Psi^{-1})^*$ is a Poisson fiberwise isomorphism between the Poisson manifolds  $V^*\alpha$ and $A^*G$.

Next, let ${\mathcal F}\tilde{L}: V\alpha \to V^*\alpha$ be the Legendre transformation 
associated with $\tilde{L}$. Then, using some results in \cite{CoLeMaMa} (see Theorem 7.6 in \cite{CoLeMaMa}), we deduce that
\begin{equation}\label{Legendre-transformations-Psi}
(\Psi^{-1})^* \circ {\mathcal F}\tilde{L} = {\mathcal F}L \circ \Psi. 
\end{equation}
Thus, from (\ref{Legendre-transformations-Psi}) and the fact that ${\mathcal F}\tilde{L}: V\alpha \to V^*\alpha$ is a fibered map 
with respect to the vector bundle projections $\tau_{V\alpha}: V\alpha \to G$ and $\tau_{V^*\alpha}: V^*\alpha \to G$, we conclude that 
${\mathcal F}\tilde{L}$ is a diffeomorphism and $\tilde{L}$ is a hyperregular Lagrangian function.

Therefore, we have the following objects:
\begin{enumerate}
\item
The \sode\ vector field $\Gamma_{\tilde{L}}$ associated with the hyperregular Lagrangian function $\tilde{L}$;
\item
The corresponding Hamiltonian function $\tilde{H}: V^*\alpha \to \R$ and the Hamiltonian vector field $X_{\tilde{H}}$ on $V^*\alpha$ and
\item
The discrete
exact Lagrangian function $\widetilde{\mathbb{L}}_h^{e}: G\alpha \to \R$ on the groupoid $G\alpha$ associated with the fibration $\alpha$. 
\end{enumerate}
Since $L$ and $\tilde{L}$ are $\Psi$-related, it follows that $H$ and $\tilde{H}$ are $(\Psi^{-1})^*$-related, which implies that
the Hamiltonian vector fields $X_H$ and $X_{\tilde{H}}$ are also $(\Psi^{-1})^*$-related (note that $(\Psi^{-1})^*$ is a Poisson morphism). Thus,
using (\ref{Legendre-transformations-Psi}), we conclude that $\Gamma_{\tilde{L}}$ is the unique \sode\  on $V\alpha$ which is $\Psi$-related with the \sode\  $\Gamma_L$ on $AG$ (see Proposition \ref{rel-SODES}).

On the other hand, note that $\Psi: V\alpha \to AG$ is just the Lie algebroid morphism associated with the Lie groupoid morphism
$\psi: G\alpha \to G$ over $\beta: G \to M$ given by
\begin{equation}\label{morph-Lie-group}
\psi(g, g') = g^{-1}g', \; \; \; \mbox{ for } (g, g') \in G\alpha.
\end{equation}
The following diagram illustrates this situation
\[
\xymatrix{ G\alpha\subset G\times G
\ar@<-0.5ex>[dd]_{pr_1|_{G\alpha}} \ar@<0.5ex>[dd]^{pr_2|_{G\alpha}}\ar[rr]^{\psi}&    &G \ar@<-0.5ex>[dd]_{\alpha} \ar@<0.5ex>[dd]^{\beta}\\
  &  & \\
G\ar[rr]^{\beta}&  & M }
\]

In conclusion, summarizing, we have the following objects:
\begin{itemize} 
\item
The Lie groupoid morphism $\psi: G\alpha \to G$ over $\beta: G \to M$ whose Lie algebroid morphism $A\psi$ is just 
the fiberwise isomorphism $\Psi: V\alpha \to AG$ between the vector bundles $V\alpha$ and $AG$;
\item
The hyperregular Lagrangian functions $L: AG \to \R$ and $\tilde{L}: V\alpha \to \R$ and the corresponding \sode\ vector fields $\Gamma_L$ and $\Gamma_{\tilde{L}}$ which 
are $\Psi$-related; 
\item
The exact discrete Lagrangian functions $\widetilde{\mathbb{L}}_h^{e}: G\alpha \to \R$ and $\mathbb{L}_h^{e}: G \to \R$ on the Lie groupoids $G\alpha$ and $G$, respectively.
\end{itemize}
Then, in order to prove Theorem \ref{exact-discrete-hamiltonian-flow}, we will find the relation between the Legendre transformation ${\mathcal F}L$ associated with
$L$ and the discrete Legendre transformations $\mathbb{F}^-\mathbb{L}^{e}_h$ and $\mathbb{F}^-\mathbb{L}^{e}_h$ associated with $\mathbb{L}^{e}_h$. For this purpose, we will discuss firstly the problem for the Lagrangian function $\tilde{L}$ and then by reduction, using $\Psi$, we will deduce the general result. In fact, we will prove some previous results in a more general setting. 

\begin{proposition}
\label{discrete-continuous-retraction}
Let $\tilde{G}\rightrightarrows \tilde{M}$ and $G\rightrightarrows M$ be Lie groupoids with Lie algebroids $\tilde{\tau}: A\tilde{G} \to \tilde{M}$ and $\tau: AG \to M$, respectively. Let $\psi: \tilde{G} \to G$ be a morphism of Lie groupoids over a map $\map{\varphi }{\tilde{M}}{M}$, and  $\map{\Psi }{A\tilde{G}}{AG}$ the induced Lie algebroid morphism. Suppose that $\tilde{\Gamma} $ and $\Gamma$ are \sode{}s on $A\tilde{G}$ and $AG$, respectively, which are $\Psi $-related, that is, $T\Psi \circ \tilde{\Gamma} =\Gamma \circ \Psi $. Then 
\[
\Psi \circ \tilde{R}^{e-}_h= R^{e-}_h\circ\psi 
\qquad\text{and}\qquad
\Psi \circ \tilde{R}^{e+}_h=R^{e+}_h\circ\psi, 
\]  
where $R^{e-}_h$ and $R^{e+}_h$ (respectively, $\tilde{R}^{e-}_h$ and $\tilde{R}^{e+}_h$) are the exact retraction maps associated with $\Gamma$ (respectively, $\tilde{\Gamma}$).
\end{proposition}
\begin{proof}
Denote by $\alpha$ and $\tilde{\alpha}$ the source maps of $G$ and $\tilde{G}$, respectively.
Let $\tilde{g}\in \tilde{G}$, and set $\tilde{m}=\tilde{\alpha} (\tilde{g})$. Denote by $\bar{\tilde{\Gamma}}$ and $\bar{\Gamma}$ the corresponding \sode\'s on $V\tilde{\alpha}$ and $V\alpha$, respectively. If $t \to \sigma_{\tilde{\varepsilon}(\tilde{m})\tilde{g}}(t)$ is the integral curve of $\bar{\tilde{\Gamma}}$ with $\sigma_{\tilde{\varepsilon}(\tilde{m})\tilde{g}}(0)=\tilde{\varepsilon}(\tilde{m})$ and $\sigma_{\tilde{\varepsilon}(\tilde{m})\tilde{g}}(h)=\tilde{g}$, then 
\[
\Psi(\tilde{R}^{e-}_h(\tilde{g}))=\Psi (\dot{\sigma}_{\tilde{\varepsilon}(\tilde{m})\tilde{g}}(0))=T\psi (\dot{\sigma}_{\tilde{\varepsilon}(\tilde{m})\tilde{g}}(0))=\frac{d}{dt}[\psi (\sigma_{\tilde{\varepsilon}(\tilde{m})\tilde{g}} (t))]\at{t=0}.
\]
On the other hand, if $\psi(\tilde{g}) = g$ and $m = \varphi(\tilde{m})$, we have that $R^{e-}_h(g)=\dot{\sigma}_{\varepsilon(m)g}(0)$, where $t \to \sigma_{\varepsilon(m)g}(t)$ is the integral curve of $\Gamma$ with $\sigma_{\varepsilon(m)g}(0)=\varepsilon(m)$ and $\sigma_{\varepsilon(m)g}(h)=g$. It follows that $\sigma_{\varepsilon(m)g}(t)=\psi (\sigma_{\tilde{\varepsilon}(\tilde{m})\tilde{g}} (t))$, because both are integral curves of $\bar{\Gamma}$ with the same initial value, $\varepsilon(m)=\psi (\tilde{\varepsilon} (\tilde{m}))$, and hence 
\[
R^{e-}_h(\psi (\tilde{g}))=\dot{\sigma}_{\varepsilon(m)g}(0)=\frac{d}{dt}[\psi (\sigma_{\tilde{\varepsilon}(\tilde{m})\tilde{g}}(t))]\at{t=0}.
\]
Therefore $\Psi (\tilde{R}^{e-}_h(\tilde{g}))=R^{e-}_h(\psi(\tilde{g}))$.

Finally, if $\Phi^{\tilde{\Gamma}}_h$ and $\Phi^{\Gamma}_h$ denote the local flows of the vector fields $\tilde{\Gamma}$ and $\Gamma$, respectively, then 
\begin{align*}
R^{e+}_h \circ \psi 
&=\Phi^{\Gamma}_h \circ R^{e-}_h\circ\psi \\
&=\Phi^{\Gamma}_h\circ \Psi \circ \tilde{R}^{e-}_h\\
&=\Psi \circ\Phi^{\tilde{\Gamma}}_h\circ\tilde{R}^{e-}_h\\
&=\Psi \circ \tilde{R}^{e+}_h\\
\end{align*}
which ends the proof.
\end{proof}
Next, we will assume that the Lie algebroid morphism $\Psi: A\tilde{G} \to AG$ is fiberwise bijective. Then, $\Psi$ induces, in a natural way,
a new fiberwise isomorphism of vector bundles
\[
(\Psi^{-1})^*: A^*\tilde{G} \to A^*G
\]
over $\varphi: \tilde{M} \to M$ between the dual bundles to $A\tilde{G}$ and $AG$, respectively.

Moreover, if $L: AG \to \R$ is a Lagrangian function on $AG$ and $\tilde{L} = L \circ \Psi$ then, using Theorem 7.6 in \cite{CoLeMaMa}, we deduce
that
\begin{equation}\label{Legendre-transformations-relation}
(\Psi^{-1})^* \circ {\mathcal F}\tilde{L} = {\mathcal F}L \circ \Psi,
\end{equation}
where ${\mathcal F}\tilde{L}: A\tilde{G} \to A^*\tilde{G}$ (respectively, ${\mathcal F}{L}: AG \to A^*G$) is the Legendre transformation associated with $\tilde{L}$ (respectively, $L$).

Now, using (\ref{Legendre-transformations-relation}) and proceeding as in the particular case when $\tilde{G}$ is the Lie groupoid associated with the source map $\alpha$ of $G$ and $\Psi$ is the canonical isomorphism between $V\alpha$ and $AG$, we also deduce that if $L$ is a regular Lagrangian function then:
\begin{enumerate}
\item
$\tilde{L}$ is also regular and
\item
The \sode\ vector fields $\Gamma_{\tilde{L}}$ and $\Gamma_{L}$ on $A\tilde{G}$ and $AG$ are $\Psi$-related.
\end{enumerate} 
In addition, we also prove
\begin{proposition}
\label{morphisms.Lagrangian}
Let  $\map{\psi }{\tilde{G}}{G}$ be a morphism of Lie groupoids and $\map{\Psi }{A\tilde{G}}{AG}$ the induced morphism of Lie algebroids. Assume that $\Psi $ is fiberwise bijective. If $L$ is a regular Lagrangian on $AG$ and $\tilde{L}=L\circ \Psi $, then $\tilde{L}$ also is regular and the corresponding exact discrete Lagrangians are related by $\mathbb{\tilde{L}}_h^{e}=\mathbb{L}_h^{e}\circ \psi $.
\end{proposition}
\begin{proof}
Since the \sode\  vector fields $\Gamma_{\tilde{L}}$ and $\Gamma_{L}$ are $\Psi$-related, it 
follows that the local flow $\Phi^{{\Gamma}_{\tilde{L}}}_t$ of the vector field $\Gamma_{\tilde{L}}$ and the local flow $\Phi^{\Gamma_L}_t$ of the vector field $\Gamma_{L}$ satisfy $\Psi \circ\Phi^{\Gamma_{\tilde{L}}}_t=\Phi^{\Gamma_L}_t \circ \Psi $. Thus   
\begin{align*}
\mathbb{L}^{e}_h(\psi (\tilde{g}))
&=\int_0^hL(\Phi_t^{\Gamma_L}(R^{e-}_h(\psi (\tilde{g}))))\,dt\\
&=\int_0^hL(\Phi_t^{\Gamma_L}(\Psi (\tilde{R}^{e-}_h(\tilde{g}))))\,dt\\
&=\int_0^hL(\Psi (\Phi_t^{\Gamma_{\tilde{L}}}(\tilde{R}^{e-}_h(\tilde{g}))))dt\\
&=\int_0^h\tilde{L}(\Phi_t^{\Gamma_{\tilde{L}}}(\tilde{R}^{e-}_h(\tilde{g})))\,dt\\
&=\tilde{\mathbb{L}}^e_h(\tilde{g}),
\end{align*}
where we have used that $R^{e-}_h\circ\psi =\Psi \circ \tilde{R}^{e-}_h$ (see Proposition \ref{discrete-continuous-retraction}).
\end{proof}
Next, we will obtain the relation between the discrete Legendre transformations associated with $\tilde{\mathbb{L}}_h^{e}$ and $\mathbb{L}^{e}_h$.
\begin{proposition}
\label{discrete-morphisms-Legendre.and.Cartan}
Let $\map{\psi }{\tilde{G}}{G}$ be a morphism of Lie groupoids and $\map{\Psi }{A\tilde{G}}{AG}$ the induced morphism of Lie algebroids. Let $L$ be a Lagrangian function on $AG$ and $\tilde{L}=L\circ\Psi $. Assume that $\Psi $ is fiberwise bijective\footnote{Alternatively, the statement of the theorem holds true if we consider any morphism such that $\mathbb{\tilde{L}}_h^{e}=\mathbb{L}_h^{e}\circ \psi$, in addition to $\tilde{L}=L\circ \Psi $}. Then 
\begin{itemize}
\item $\F^+\mathbb{L}_h^{e} \circ\psi =(\Psi^{-1})^* \circ \F^+\mathbb{\tilde{L}}^e_h$,
\item $\F^-\mathbb{L}_h^{e} \circ\psi =(\Psi^{-1})^* \circ \F^-\mathbb{\tilde{L}}^e_h$.
\end{itemize}
\end{proposition}
\begin{proof}
Let $\tilde{g}$ be a point in $\tilde{G}$ and $\tilde{a} \in A\tilde{G}$ such that $\tilde{\alpha}(\tilde{g})=\tilde{\tau}(\tilde{a})$. We take a curve $\tilde{\sigma}$ in the $\tilde{\alpha}$-fiber with derivative $\tilde{a}$ at $s =0$.
Then, 
\begin{align*}
\pai{(\F^+\mathbb{L}_h^{e} \circ\psi )(\tilde{g})}{\Psi(\tilde{a})}
&=\frac{d}{ds}\mathbb{L}_h^{e}(\psi (\tilde{\sigma}(s)))\at{s=0}\\
&=\frac{d}{ds}\mathbb{\tilde{L}}_h^{e}(\tilde{\sigma}(s))\at{s=0}\\
&=\pai{\F^+\mathbb{\tilde{L}}_h^{e}(\tilde{g})}{\tilde{a}}\\
& = \pai{((\Psi^{-1})^* \circ \F^+\mathbb{\tilde{L}}^e_h)(\tilde{g})}{\Psi(\tilde{a})}\\ 
\end{align*}
where we have used Proposition \ref{morphisms.Lagrangian} and the fact that $\psi (\tilde{\sigma} (s))$ is a curve in the $\alpha $-fiber with derivative $\Psi (\tilde{a})$ at $s=0$.

Similarly,
\begin{align*}
\pai{(\F^-\mathbb{L}^{e}_h\circ\psi )(\tilde{g})}{\Psi(\tilde{a})}
&=-\frac{d}{ds}\mathbb{L}_h^{e}(\psi (\tilde{\sigma} (s)^{-1}))\at{s=0}\\
&=-\frac{d}{ds}\mathbb{\tilde{L}}_h^{e}(\tilde{\sigma}(s)^{-1})\at{s=0}\\
&=\pai{\F^-\mathbb{\tilde{L}}_h^{e}(\tilde{g})}{\tilde{a}}\\
& = \pai{((\Psi^{-1})^* \circ \F^-\mathbb{\tilde{L}}^e_h)(\tilde{g})}{\Psi(\tilde{a})}\\ 
\end{align*}
which completes the proof.
\end{proof}

Now, we obtain the relation between the Legendre transformation of a regular continuous Lagrangian function and the Legendre transformations of the corresponding exact discrete Lagrangian function.

\begin{theorem}\label{Legendre-transformation-cont-discrete}
The Legendre transformation $\calf L$ of a regular continuous Lagrangian and the Legendre transformations $\F^\pm\Le^{e}$ 
of the corresponding exact discrete Lagrangian are related by
\[
\F^-\Le^{e}=\calf L\circ R^{e-}_h\qquad\text{and}\qquad\F^+\Le^{e}=\calf L\circ R^{e+}_h.
\] 
\end{theorem}
\begin{proof}
We will prove that the result holds true for a Lagrangian system in the Lie groupoid $G\alpha $, and later we will extend it to an arbitrary Lie groupoid by using reduction. 

Let $\tilde{L}$ be a continuous Lagrangian on $V\alpha $ and $\mathbb{\tilde{L}}^{e}_h$ the corresponding discrete exact Lagrangian. Let $\Gamma_{\tilde{L}}$ be the \sode\ solution of the continuous dynamics and $\phi_t^{\Gamma_{\tilde{L}}}$ its local flow. 

From the results in~\cite{Ma} we have that if $a(s,t)$ is a variation of a curve $a(0,t)$ which is a solution of the Euler-Lagrange equations for a Lagrangian function $L:A \to \Real$ on a Lie algebroid $A$ over $M$, and $b(s,t)$ is the corresponding infinitesimal variation (i.e. $\Upsilon(s,t)=a(s,t)dt+b(s,t)ds$ is a morphism of Lie algebroids from $T\Real^2$ to $A$) with either $b(s,t_0)=0$ or $b(s,t_1)=0$, then   
\[
\frac{d}{ds}S(a_s)\at{s=0}
=\int_{t_0}^{t_1}\frac{d}{ds}L(a(s,t))\at{s=0}\ \,dt 
=\pai{\calf L(a(0,t))}{b(0,t)}\Big|_{t_0}^{t_1}
\]
In the case of the Lie groupoid $G\alpha $ and its Lie algebroid $V\alpha $, we have that the infinitesimal variations are of the form 
\[
\Upsilon(s,t)=T_{(s,t)}\gamma =\pd{\gamma }{t}(s,t)dt+\pd{\gamma }{s}(s,t)ds,
\]
where $\gamma (s,t)$ takes values on a fiber of $\alpha:G \to M$. Thus, $a(s,t)=\pd{\gamma }{t}(s,t)$ and $b(s,t)=\pd{\gamma }{s}(s,t)$, both taking values in $V\alpha $.

We will apply the above result for $t_0=0$ and $t_1=h$. Given a vector $v_0\in V_{g_0}\alpha $, we consider a curve $s \to g(s)$ such that, $g(0) = g_0$, $\alpha (g(s))=\alpha (g_0)$ and $\dot{g}(0)=v_0$. Define the variation  $\gamma (s,t)=\tau_{V\alpha}(\Phi _t^{\Gamma_{\tilde{L}}}(\tilde{R}_h^{e-}(g(s),g_1)))$, so that 
\[
a(s,t)=\pd{\gamma }{t}(s,t)=\phi _t^{\Gamma_{\tilde{L}}}(\tilde{R}_h^{e-}(g(s),g_1))
\qquand
b(s,t)=\pd{\gamma }{s}(s,t).
\]
Then we have
\begin{align*}
&\text{at $t=0$,}&&\gamma (s,0)=\tau_{V\alpha}(\Phi _0^{\Gamma_{\tilde{L}}}(\tilde{R}_h^{e-}(g(s),g_1)))=\tau_{V\alpha}(\tilde{R}_h^{e-}(g(s),g_1))=g(s)\\
&\text{at $t=h$,}&&\gamma (s,h)=\tau_{V\alpha}(\Phi _h^{\Gamma_{\tilde{L}}}(\tilde{R}_h^{e-}(g(s),g_1)))= \tau_{V\alpha}(\tilde{R}_h^{e+}(g(s),g_1))=g_1,
\end{align*}
so that 
\[
b(s,0)=\dot{g}(s)
\qquad\text{and}\qquad
b(s,h)=0.
\]
Therefore, for $(g_0,g_1)\in G\alpha $, 
\begin{align*}
\pai{\F^-\mathbb{\tilde{L}}^{e}_h(g_0,g_1)}{v_0}
&=-\frac{d}{ds}\mathbb{\tilde{L}}_h^{e}\bigl((g_0,g(s))^{-1}(g_0,g_1)\bigr)\at{s=0}\\
&=-\frac{d}{ds}\mathbb{\tilde{L}}_h^{e}(g(s),g_1)\at{s=0}\\
&=-\frac{d}{ds}\int_0^h \tilde{L}(\Phi_t^{\Gamma_{\tilde{L}}}(\tilde{R}^{e-}_h(g(s),g_1)))dt\at{s=0}\ \\
&=\pai{\calf {\tilde{L}}(a(0, 0))}{b(0,0)}-\pai{\calf {\tilde{L}}(a(0, h))}{b(0,h)}\\
&=\pai{\calf {\tilde{L}}(\tilde{R}^{e-}_h(g_0,g_1))}{v_0}.
\end{align*}
This proves the first relation for the groupoid $G\alpha $.

For the second relation, we proceed as follows. If $v_1 \in V_{g_1}\alpha$, we take a curve $s \to g(s)$
on $G$ such that $g(0) = g_1$, $\alpha(g(s)) = \alpha(g_1)$ and $\dot{g}(0)=v_1$. Then,
 we define the variation  $\gamma (s, t)=\tau_{V\alpha}(\Phi_t^{\Gamma_{\tilde{L}}}(\tilde{R}^{e-}_h(g_1,g(s))))$, so that 
\[
a(s,t)=\pd{\gamma }{t}(s,t)=\Phi _t^{\Gamma_{\tilde{L}}}(\tilde{R}^{e-}_h(g_1,g(s)))
\qquand
b(s,t)=\pd{\gamma }{s}(s,t).
\]
Thus, we have
\begin{align*}
&\text{at $t=0$,}&&\gamma (s,0)=\tau_{V\alpha}(\Phi _0^{\Gamma_{\tilde{L}}}(\tilde{R}_h^{e-}(g_1, g(s))))=\tau_{V\alpha}(\tilde{R}_h^{e-}(g_1,g(s)))=g_1\\
&\text{at $t=h$,}&&\gamma (s,h)=\tau_{V\alpha}(\Phi _h^{\Gamma_{\tilde{L}}}(\tilde{R}_h^{e-}(g_1, g(s))))=\tau_{V\alpha}(\tilde{R}_h^{e+}(g_1,g(s)))=g(s)
\end{align*}
and therefore
\[
b(s, 0)=0
\qquad\text{and}\qquad
b(0, h)=\dot{g}(0)=v_1.
\]
Consequently, for $(g_0,g_1)\in G\alpha $,
\begin{align*}
\pai{\F^+\mathbb{\tilde{L}}^{e}_h(g_0,g_1)}{v_1}
&=\frac{d}{ds}\mathbb{\tilde{L}}^{e}_h\bigl((g_0, g(s))\bigr)\at{s=0}\\
&=\frac{d}{ds}\int_0^h \tilde{L}(\Phi _t^{\Gamma_{\tilde{L}}}(\tilde{R}^{e-}_h(g_0,g(s)))) dt \at{s=0}\ \\
&=\pai{\calf {\tilde{L}}(\tilde{R}^{e+}_h(g_0,g_1))}{v_1}-\pai{\calf {\tilde{L}}(\tilde{R}^{e-}_h(g_0,g_1))}{b(0,0)}\\
&=\pai{\calf {\tilde{L}}(\tilde{R}^{e+}_h(g_0,g_1))}{v_1}.
\end{align*}
This proves the statement for the case of the Lie groupoid $G\alpha $.

For the general case, consider the Lie groupoid morphism $\map{\psi }{G\alpha }{G}$ over $\map{\varphi =\beta }{G}{M}$ given by
\[
\psi (g_0,g_1)=g_0^{-1}g_1.
\]
The induced Lie algebroid morphism $\map{\Psi }{V\alpha }{AG}$ is
\[
\Psi(v)=T_gl_{g^{-1}}(v), \; \; \mbox{ for } v \in V_g\alpha \mbox{ and } g \in G.
\]
If $L \in C^{\infty}(AG)$ is a regular Lagrangian, we define the Lagrangian $\tilde{L}\in C^\infty (V\alpha )$ given by $\tilde{L}(v)=L(T_gl_{g^{-1}}(v))$, i.e. $\tilde{L}=L\circ\Psi $. It follows, from Proposition~\ref{morphisms.Lagrangian}, that the corresponding exact Lagrangians are related by $\mathbb{\tilde{L}}^{e}_h=\mathbb{L}^{e}_h\circ\psi $ and, from Proposition~\ref{discrete-morphisms-Legendre.and.Cartan}, we have 
\[
\F^\pm \mathbb{L}^{e}_h \circ \psi=(\Psi^{-1})^* \circ \F^\pm \mathbb{\tilde{L}}^{e}_h.
\]
From $\F^\pm\mathbb{\tilde{L}}^{e}_h  =\calf {\tilde{L}}\circ\tilde{R}^{e\pm} _h$ we get 
\[
\F^\pm\mathbb{{L}}^{e}_h \circ\psi = (\Psi^{-1})^* \circ \calf \tilde{L} \circ \tilde{R}^{e\pm}_h.
\]
On the other hand, using that $L$ and $\tilde{L}$ are $\Psi$-related, we have that
 $(\Psi^{-1})^* \circ \calf {\tilde{L}}= \calf {L}\circ\Psi $, so that we finally get
\[
\F^\pm \mathbb{L}^{e}_h \circ \psi =  \calf L \circ \Psi \circ \tilde{R}^{e\pm}_{h} = \calf L \circ R^{e\pm}_h \circ \psi,
\]
where we have used that $\Psi \circ \tilde{R}^{e\pm}_h = R^{e\pm}_h \circ \psi$ (see Proposition \ref{discrete-continuous-retraction}).
The result follows by noticing that $\psi $ is surjective.
\end{proof}
Next, we will prove Theorem \ref{exact-discrete-hamiltonian-flow}.
\begin{proof}
(i) Using Theorem \ref{Legendre-transformation-cont-discrete}, it follows that $\F^\pm \mathbb{{L}}^{e}_h$ are local diffeomorphisms. This proves the result. 

(ii) From (\ref{retraction-flow2}) and Theorem \ref{Legendre-transformation-cont-discrete}, we deduce that
\[
\F^+\mathbb{L}^{e}_h =\calf L\circ R^{e+}_h=\calf L\circ\Phi^{\Gamma_L}_h\circ R^{e-}_h=\Phi^{X_H}_h\circ\calf L\circ\ R^{e-}_h=\phi^{X_H}_h\circ\F^-\mathbb{L}^{e}_h,
\]
where we have used that $\calf L\circ\Phi^{\Gamma_L}_h=\Phi^{X_H}_h\circ\calf L$.
\end{proof}

\subsection{Variational error analysis}
\label{var-error-anal}
In this section, we will extend the results of Patrick and Cuell \cite{PaCu} (see also \cite{MaWe}) on the variational error analysis for the more general case when the continuous regular Lagrangian function is defined on the Lie algebroid of a Lie groupoid.

For this purpose, we will use the notion of an approximation order in the setting of smooth manifolds which was introduced in \cite{PaCu}.

\medskip

\noindent {\bf Approximation order}

Let $F_i: I \times Q \to Q'$, $i = 1,2$, with $I$ an open interval in $\mathbb{R}$ ($0 \in I$) and $Q$, $Q'$ smooth manifolds, such that $F_1(0, q) = F_2(0, q)$, for all $q \in Q$, then
\[
F_1 = F_2 + O(h^r), \; \; \mbox{ with } r \geq 1,
\] 
if for every $q \in Q$ there exists a local chart $\varphi$ at $q' = F_i(0, q)$ in $Q'$ such that
\begin{equation}\label{order-approx}
\varphi(F_2(h, \bar{q})) - \varphi(F_1(h, \bar{q})) = h^r ((\delta F)_{\varphi}(h, \bar{q}))
\end{equation}  
for $(h, \bar{q})$ in some neighborhood of $(0, q)$, where $(\delta F)_{\varphi}$ is a smooth function in such a neighborhood. This definition does not depend on the coordinate chart (see \cite{PaCu}).

Note that (\ref{order-approx}) implies the usual condition
\[
\| \varphi(F_2(h, \bar{q})) - \varphi(F_1(h, \bar{q})) \| \leq C |h|^r
\]
in some neighborhood of $(0, q)$, with $C$ a positive constant.

\medskip

\noindent{\bf The particular case when $AG = V\pi$}

Next, we will discuss the particular case when the regular continuous Lagrangian function is defined in the vertical bundle $V\pi$ of  a fibration $\pi: P \to M$.

Let $\tilde{L}: V\pi \to \mathbb{R}$ be a regular Lagrangian function.

Denote by $\mathbb{\tilde{L}}^{e}_h $ the exact discrete Lagrangian function on the Lie groupoid $G\pi$ associated with $\tilde{L}$.

As in Section \ref{exact-discrete-Lie-groupoid} and in order to simplify the reading, we will omit any reference to the domain of the maps that we are going to consider, and we will just indicate the spaces where they are defined.

Now, let $\tilde{L}_h^d$ be a regular discrete Lagrangian function on $G\pi$ such that
\[
\tilde{L}_h^d = \mathbb{\tilde{L}}^{e}_h + O(h^{r+1}).
\]
In this case, we say that the given discrete Lagrangian  $\tilde{L}_h^d $ is of order $r$ \cite{MaWe}. 
Then, since our discussion is local, we may consider a regular Lagrangian function $\bar{L}$ on $TP$ such that:
\begin{enumerate}
\item
It is a (local) extension of $\tilde{L}$ and

\item
The (local) restriction to $V\pi$ of the \sode\ $\Gamma_{\bar{L}}$ on $TP$ associated with $\bar{L}$ is just the \sode\ $\Gamma_{\tilde{L}}$ on $V\pi$ associated with $\tilde{L}$.
\end{enumerate}
Thus, if $i_{G\pi}: G\pi \to P \times P$ is the canonical inclusion and $F_{{\mathbb{\bar{L}}}^{e}_h}$ (resp., $F_{{\mathbb{\tilde{L}}}^{e}_h}$) is the discrete Lagrangian evolution operator for ${\mathbb{\bar{L}}}^{e}_h$ (resp.,
${\mathbb{\tilde{L}}}^{e}_h$), we have that
\begin{equation}\label{comp-flujos-exact1} 
F_{{\mathbb{\tilde{L}}}^{e}_h} = F_{{\mathbb{\bar{L}}}^{e}_h} \circ i_{G\pi}.
\end{equation}
In a similar way, we may consider a regular discrete Lagrangian function $\bar{L}^d_h$ on the Lie groupoid $P \times P$ such that:
\begin{enumerate}
\item
It is a (local) extension of $\tilde{L}^d_h$;

\item
We have that
\[
\bar{L}^d_h = {\mathbb{\bar{L}}}^{e}_h + O(h^{r+1})
\]
and
\item
If $F_{\bar{L}^d_h}$ and $F_{\tilde{L}_h^d}$ are the discrete Lagrangian evolution operators associated with $\bar{L}_h^d$ and $\tilde{L}_h^d$, respectively, then
\begin{equation}\label{comp-flujos-exact2}
F_{\tilde{L}^{d}_h} = F_{\bar{L}^{d}_h} \circ i_{G\pi}.
\end{equation}
\end{enumerate}
Now, denote by $\bar{R}^{e-}_{h}$ the exact retraction from $P \times P$ on $TP$ associated with $\bar{L}$. Then, one may consider the evolution operators on $TP$ given by
\begin{equation}\label{Evolution-tangent}
\bar{F}_{\bar{L}^d_h} = \bar{R}_{h}^{e-} \circ F_{\bar{L}_h^d} \circ (\bar{R}_{h}^{e-})^{-1}, \; \; \bar{F}_{{\mathbb{\bar{L}}}^{e}_h}= \bar{R}_{h}^{e-}\circ F_{{\mathbb{\bar{L}}}^{e}_h} \circ (\bar{R}_{h}^{e-})^{-1}.
\end{equation}

In addition, using Theorem 4.8 in \cite{PaCu}, we deduce that
\[
\bar{F}_{\bar{L}^d_h}  = \bar{F}_{{\mathbb{\bar{L}}}^{e}_h} + O(h^{r+1}).
\]
Thus, from (\ref{Evolution-tangent}) and Proposition 4.4 in \cite{PaCu}, we obtain that
\[
F_{\bar{L}^d_h}  = F_{{\mathbb{\bar{L}}}^{e}_h} + O(h^{r+1}).
\]
Therefore, using (\ref{comp-flujos-exact1}) and (\ref{comp-flujos-exact2}), it follows that
\[
F_{\tilde{L}^d_h}  = F_{{\mathbb{\tilde{L}}}^{e}_h} + O(h^{r+1}).
\]
The following diagram illustrates the previous situation:
\[
\xymatrix{ G\pi\ar[rr]^{i_{G\pi}}\ar[dd]^{F_{\tilde{L}_h^d}}&   &P\times P\ar[dd]^{F_{\bar{L}_h^d}}\ar[rr]^{\bar{R}_{h}^{e-}}&&TP\ar[dd]^{{\bar{F}_{\bar{L}_h^d}}}\\
  &  & & &\\
  G\pi\ar[rr]^{i_{G\pi}}&   &P\times P\ar[rr]^{\bar{R}_{h}^{e-}}&&TP}
\]
In conclusion, we have proved the following result

\begin{proposition}\label{order-approx-particular}
Let $\pi: P \to M$ be a fibration and $\tilde{L}: V\pi \to \mathbb{R}$ a regular Lagrangian function. If $\tilde{L}^d_h$ is an order $r$ discretization on the Lie groupoid $G\pi$ then,
\[
F_{\tilde{L}^d_h}  = F_{{\mathbb{\tilde{L}}}^{e}_h} + O(h^{r+1}),
\]
where $F_{\tilde{L}^d_h}$ is the discrete Lagrangian evolution operator for $\tilde{L}^d_h$ and $F_{{\mathbb{\tilde{L}}}^{e}_h}$ is the exact discrete Lagrangian evolution operator.
\end{proposition}

\medskip

\noindent {\bf The general case}

Now, we will extend Proposition \ref{order-approx-particular} for the general case when the continuous Lagrangian function is defined on the Lie algebroid of an arbitrary Lie groupoid.

In fact, we will prove the following result

\begin{theorem}\label{theorem-error}
Let $L: AG \to \mathbb{R}$ be a regular Lagrangian function on the Lie algebroid of a Lie groupoid $G$ over $M$. Suppose that $L^d_h: G \to  \mathbb{R}$ is a regular discrete Lagrangian function on $G$ and that  ${\mathbb{L}}^{e}_h$ is the exact discrete Lagrangian function on $G$ associated with $L$. If $L^d_h$ is an order $r$ discretization then
\[
F_{L^d_h}  = F_{{\mathbb{L}}^{e}_h} + O(h^{r+1}),
\]
where $F_{L^d_h}$ is the discrete Lagrangian evolution operator for $L^d_h$ and $F_{{\mathbb{L}}^{e}_h}$ is the exact discrete Lagrangian evolution operator.
\end{theorem}
\begin{proof}
Denote by $\psi: G\alpha \to G$ the Lie groupoid morphism given by (\ref{morph-Lie-group}) and by $\Psi: V\alpha \to AG$ the corresponding Lie algebroid morphism.

Now, let $\tilde{L}: V\alpha \to \mathbb{R}$ be the Lagrangian function on $V\alpha$ defined by
\[
\tilde{L} = L \circ \Psi.
\]
Then, from Proposition \ref{morphisms.Lagrangian}, we deduce that $\tilde{L}$ is a regular Lagrangian function on $V\alpha$ and 
\begin{equation}\label{comparacion-exact}
{\mathbb{\tilde{L}}}^{e}_h = {\mathbb{L}}^{e}_h \circ \psi,
\end{equation}
${\mathbb{\tilde{L}}}^{e}_h$ being the exact discrete Lagrangian function on $G\alpha$ associated with $\tilde{L}$.

Next, we consider the discrete Lagrangian function $\tilde{L}^d_h$ on $G\alpha$ given by
\begin{equation}\label{comparacion-discretos}
\tilde{L}_h^d = L_h^d \circ \psi.
\end{equation}
Using (\ref{comparacion-exact}), (\ref{comparacion-discretos}) and the fact that $L^d_h$ is an order $r$ discretization of $L$, we obtain that $\tilde{L}^d_h$ also is an order $r$ discretization of $\tilde{L}$. Thus, from Proposition \ref{order-approx-particular}, it follows that
\begin{equation}\label{comparacion-flujos-tilde}
F_{\tilde{L}^d_h}  = F_{{\mathbb{\tilde{L}}}^{e}_h} + O(h^{r+1}).
\end{equation}
On the other hand, using (\ref{comparacion-exact}), (\ref{comparacion-discretos}), Corollary 4.7 in \cite{MaMaMa} and the fact that $\psi$ is an epimorphism of Lie groupoids, we deduce that
\[
F_{L^d_h} \circ \psi = \psi \circ F_{\tilde{L}^d_h}, \; \; \; F_{{\mathbb{L}}^{e}_h} \circ \psi = \psi \circ F_{{\mathbb{\tilde{L}}}^{e}_h}.
\]
Therefore, from (\ref{comparacion-flujos-tilde}) and Proposition 4.4 in \cite{PaCu}, we conclude that
\[
F_{L^d_h} = F_{{\mathbb{L}}^{e}_h} + O(h^{r+1}),
\]
which proves the result.
\end{proof}


\subsection{Example: Standard lagrangians defined on the tangent bundle of a manifold. Discretizations of the standard Euler-Lagrange equations}\label{example1}

Let $L: TQ \to \mathbb{R}$ be an standard regular Lagrangian function defined on the tangent bundle of a manifold $Q$. $TQ$ is the Lie algebroid of the pair (banal) groupoid $Q \times Q$ over $Q$ (see Appendix \ref{algebroide-grupoide}).

As we know, the Euler-Lagrange equations for $L$ are 
\[
\displaystyle \frac{d}{dt}(\frac{\partial L}{\partial \dot{q}^{i}}) - \frac{\partial L}{\partial q^{i}} = 0.
\]
Moreover, if $q$ is a point of $Q$, then we can consider the exact retraction $R^{e-}_{h}$ associated with the standard \sode\ $\Gamma_L$ which is a diffeomorphism,
\[
R^{e-}_{h}: U \times U \subseteq Q \times Q \to {\mathcal U}\subseteq TQ,
\]
for a sufficiently small positive number $h$, an open subset $U \subseteq Q$, $q \in U$, and an open subset ${\mathcal U} \subseteq TQ$ (see Theorem \ref{convexity2}).

In addition, using (\ref{exact-lagrangian}), we deduce that the exact discrete Lagrangian function $\mathbb{L}_h^{e}$ associated with $L$ is defined by
\begin{equation}\label{exact-standar}
\mathbb{L}^{e}_h(q_0, q_1) = \int_{0}^{h} L(q(t), \dot{q}(t)) dt
\end{equation}
for $(q_0, q_1) \in U \times U \subseteq Q \times Q$, where $t \to q(t)$ is the unique solution of the Euler-Lagrange equations which satisfies
\[
q(0) = q_0, \; \; \; \; q(h) = q_1.
\]
Note that (\ref{exact-standar}) is the well-known expression of the exact discrete Lagrangian function associated with $L$ (see, for instance, \cite{MaWe}).

Now, our aim is to derive  a discretization for this exact discrete Lagrangian. Therefore,  we will use an auxiliary  riemannian metric ${\mathcal G}$ on $Q$ with associated geodesic spray 
$\Gamma_{\mathcal G}$. The associated exponential  for an enough small $h_0>0$ is
\[
exp^{\Gamma_\mathcal G}_{h_0}: D_{h_0}^{\Gamma}\subseteq TQ\rightarrow Q\times Q
\]
as in Section 3. 
Observe that 
\[
exp^{\Gamma_{\mathcal G}}_{h_0}(v)=(\tau_Q(v), exp_{\tau_Q(v)}(h_0v))
\]
where the last one is the standard exponential map on a riemannian manifold defined by
\[
exp_{\tau_Q(v)}(v)=\gamma_v(1)
\]
where $t\rightarrow \gamma_v(t)$ is the unique geodesic  such that $\gamma_v'(0)=v$.

In our case, it will be useful to use the following map (see Remark \ref{reweyls})
\begin{equation}\label{exp-tilde-h0}
\widetilde{exp}^{\Gamma_{\mathcal G}}_{h_0}(v)=(\gamma_v(-\frac{h_0}{2}), \gamma_v(\frac{h_0}{2}))\; .
\end{equation}
In the case of the Euclidean metric in $M={\mathbb R}^n$ we have that
\begin{equation}\label{exp-tilde-h0-est}
\widetilde{exp}^{\Gamma_{\mathcal G}}_{h_0}(q, u)=(q-\frac{h_0}{2}u,q+\frac{h_0}{2}u)\; .
\end{equation}
 Thus, a discretization of $\mathbb{L}^e_h$ would be
 \[
L^d_h(q_0, q_1)=h L((\widetilde{exp}^{\Gamma_{\mathcal G}}_{h})^{-1}(q_0, q_1))\; .
 \]
 In the case of $M={\mathbb R}^n$ we have that 
  \[
L^d_h(q_0, q_1)=h L(\frac{q_0+q_1}{2}, \frac{q_1-q_0}{h})\; .
 \]
 
\subsection{Example: Lagrangians defined on a Lie algebra. Discretizations of the Euler-Poincar\'e equations}

A particular case of our framework corresponds when the Lie algebroid $A$ is a Lie algebra ${\mathfrak g}$, a Lie algebroid over a single point. In such a case, we can consider a connected Lie group $G$ whose Lie algebra is ${\mathfrak g}$. $G$ is a Lie groupoid over a single point and the Lie algebroid of $G$ may be identified, in a natural way, with ${\mathfrak g}$ (see Appendix \ref{algebroide-grupoide}). Moreover, the Lie groupoid morphism $\psi: G\alpha \to G$ considered in the previous sections is just the map
\[
\psi: G \times G \to G, \; \; \; (g_0, g_1) \to g_0^{-1}g_1.
\]
$\psi$ is a morphism between the pair groupoid $G \times G$ and the Lie group $G$. The corresponding Lie algebroid morphism $\Psi: TG \to {\mathfrak g}$ is just the map
\[
\Psi(g, \dot{g}) = (T_gl_{g^{-1}})(\dot{g}) = g^{-1}\dot{g}.
\]
Now, suppose that ${\mathfrak l}: {\mathfrak g} \to \mathbb{R}$ is a hyperregular Lagrangian function on ${\mathfrak g}$.  
Then, the corresponding Euler-Lagrange equations for this function are the well-known Euler-Poincar\'e equations (see \cite{CeMaRa}): 
\begin{equation}\label{euler-poincare}
\frac{d}{dt}\left(\frac{\partial{\mathfrak l}}{\partial \xi}\right)=\hbox{ad}^*_{\xi}\frac{\partial {\mathfrak l}}{\partial \xi}\; .
\end{equation}
As in Sections \ref{exact-discrete-Lie-groupoid} and \ref{var-error-anal}, we can consider the standard (left-invariant) regular Lagrangian function $L = {\mathfrak l} \circ \Psi: TG \to \mathbb{R}$, that is,
\[
L (g, \dot{g}) = {\mathfrak l}(g^{-1}\dot{g}).
\]
It is well-known that if $\xi: I \subseteq \mathbb{R} \to {\mathfrak g}$ is a solution of the Euler-Poincar\'e equations then there exists a unique solution 
\[
t \to (g(t), \dot{g}(t))
\]
of the Euler-Lagrange equations for $L$ satisfying
\[
g(0) = {\mathfrak e} \mbox{ and } \dot{g}(t) = g(t) \xi(t) = (T_{\mathfrak e}l_{g(t)})(\xi(t)), \; \; \; \mbox{ for every } t \in I,
\]
where ${\mathfrak e}$ is the identity element in $G$ (see \cite{CeMaRa}).

Thus, from (\ref{exact-lagrangian-1}), we have that if $h > 0$ is sufficiently small and $U$ is a suitable open neighborhood of ${\mathfrak e}$, the exact discrete Lagrangian function associated with ${\mathfrak l}$ is given by
\[
{\mathfrak l}_h^{e}(g) = \int_{0}^{h} {\mathfrak l}(\xi_g(t))dt, \; \; \; \mbox{ for } g \in U,
\]
where $\xi_g: I \subseteq \mathbb{R} \to {\mathfrak g}$ is the unique solution of the Euler-Poincar\'e equations for ${\mathfrak l}: {\mathfrak g} \to \mathbb{R}$ such that the corresponding solution $(g, \dot{g}): I \subseteq \mathbb{R} \to TG$ of the Euler-Lagrange equations for $L$ satisfies
\[
g(0) = {\mathfrak e}, \; \; \; g(h) = g.
\]    
Therefore, we have that
\begin{eqnarray}\label{depe}
{\mathfrak l}_h^{e}(g) =\int_0^h{\mathfrak l}(\xi_g(t))\, dt=\int^h_0 L(g(t), \dot{g}(t))\, dt 
&=& \int^h_0 \tilde{L}_{exp}(\eta(t), \dot{\eta}(t))\, dt 
\end{eqnarray}
where 
$\eta(t)={exp}^{-1}(g(t))$ and $\dot{\eta}(t)=T_{g(t)} {exp}^{-1} (\dot{g}(t))$ and 
$\tilde{L}_{exp}: T{\mathfrak g}\equiv {\mathfrak g}\times \mathfrak{g}\longrightarrow {\mathbb R}$ is defined by $\tilde{L}_{ exp}(\eta(t), \dot{\eta}(t))=L(g(t), \dot{g}(t))$. 
Here, ${exp}: {\mathfrak g}\rightarrow { G}$ is the exponential map which for  a finite-dimensional Lie group $G$ is a local diffeomorphism. 
The following diagram illustrates the situation: 
\[
\xymatrix{ 
{\mathbb R}&T{\mathfrak g}\ar[l]_{\tilde{L}_{exp}}
\ar[d]_{\tau_{\mathfrak g}} \ar[rr]^{T{exp}}&    &TG \ar[d]_{\tau_G}\ar[r]^{L}&{\mathbb R} \\
&{\mathfrak g}\ar[rr]^{{{exp}}}&  & G& }
\]

Now, since $\tilde{L}_{exp}: T{\mathfrak g}\rightarrow {\mathbb R}$ is a Lagrangian defined on a tangent bundle of a vector space is quite simple to consider discretizations by truncating the last integral in (\ref{depe}) as, for instance, in the procedures given 
in \cite{Leok,LeSh,MaWe}.

For instance, we start analyzing the following typical  discretization of the action where $\alpha\in[0,1]$
\begin{eqnarray*}
{\mathfrak l}_h^{d, \alpha}(g)&=&h\tilde{L}_{exp}(\alpha{\eta}, {\eta}/h)\\
&=&hL(exp(\alpha\eta), T_{\alpha\eta}exp(\eta/h))\\
&=&h{\mathfrak l}(d_l exp_{\alpha\eta}(\eta/h))
\end{eqnarray*}
where $\eta=exp^{-1}(g)$ and $\mbox{d}_l {exp}_{\xi}:\mathfrak{g}\rightarrow\mathfrak{g}$ is  defined  as follows
  \begin{eqnarray*}
    T_{\xi}{exp} (\eta)&=&T_{\frak e} l_{{exp}(\xi)}\left(\mbox{d}_l{exp}_{\xi}(\eta)\right)={exp}(\xi)\mbox{d}_l{exp}_{\xi}(\eta)\, , \ \hbox{for all}\    \xi, \eta\in\mathfrak{g}\; ,
    \end{eqnarray*}
and it is called the left-trivialized derivative of the map $exp$.
 Using  Theorem 1.7, Chapter II, in \cite{Helga} (see also \cite{HLW,IMNZ}), we have  that 
 \[
 \mbox{d}_l{exp}_{\xi}(\eta)=\frac{1-e^{-ad_{\xi}}}{ad_{\xi}}(\eta)=\sum_{j=0}^{\infty}\frac{(-1)^j ad_\xi^j\eta}{(j+1)!}\; .
 \]
 Therefore, $d_l exp_{\alpha\eta}(\eta/h)=\eta/h$ and, in consequence, 
\[
{\mathfrak l}_h^{d, \alpha}(g)=h{\mathfrak l}(\frac{{exp}^{-1}(g)}{h})\; ,
 \]
 and it is independent of the parameter $\alpha\in [0,1]$ (see also \cite{MaPeSh}). 
 Starting with  the midpoint discretization on $T{\mathfrak g}$, that is, $\alpha=1/2$, automatically the associated  discrete Lagrangian evolution operator is second-order, that is, 
\[
F_{{\mathfrak l}_h^{d, \alpha}}=F_{{{\mathfrak l}}_h^e}+O(h^3)
\]
 Observe that the unique role played by the exponential map is to locally  transform the Lie group $G$ into the vector space ${\mathfrak g}$. Indeed, we can use a different map for this task. This is typically accomplished using  a local
diffeomorphism $\tau: {\mathfrak g}\to G$ such that $\tau(0)={\frak e}$ (this is, for instance, the case of retraction maps \cite{Rabee}). 
Using $\tau$ instead of $exp$ we obtain the alternative  discretization
\begin{equation}\label{rew}
{\mathfrak l}_h^{d, \tau, \alpha}(g)=h{\mathfrak l}(d_l\tau_{\alpha\tau^{-1}(g)}(\tau^{-1}(g)/h)))\; .
 \end{equation}
where $\alpha\in [0,1]$ and
the left-trivialized derivative $\mbox{d}_l\tau_{\xi}:\mathfrak{g}\rightarrow\mathfrak{g}$ is defined   for all   $\eta\in\mathfrak{g}$ as follows
  \begin{eqnarray*}
    T_{\xi}\tau (\eta)&=&T_{\frak e} l_{\tau(\xi)}\left(\mbox{d}_l\tau_{\xi}(\eta)\right)=\tau(\xi) \mbox{d}_l\tau_{\xi}(\eta).
      \end{eqnarray*}
Another option is to consider the symmetrized discrete Lagrangian
\[
{\mathfrak l}_h^{sym,\tau, \alpha} =\frac{1}{2} {\mathfrak l}_h^{d, \tau, \alpha}+ \frac{1}{2}{\mathfrak l}_h^{d, \tau, 1-\alpha}
\]
 giving a method that is second-order for any $\alpha\in [0,1]$ (see \cite{MaWe}). It is a Lie group version of the classical  St\"ormer-Verlet method when $\alpha=0$.

Applying the same idea it is possible to derive new discretizations  based on 
symplectic partitioned Runge-Kutta methods \cite{BoMa,Mart}. Indeed, consider the following discrete lagrangian
\[
{\mathfrak l}^{RK,\tau}_h(g)=h\sum_{i=1}^s b_i {\mathfrak l}(d_l\tau_{\xi_i}(\eta_i)))
\]
where $\xi_i=h\sum_{j=1}^s a_{ij}\eta_j$ and where $\eta_i$ are chosen to extremize ${\mathfrak l}^{RK,\tau}_h$  under the constraint  $\tau^{-1}(g)=h\sum_{j=1}^s b_j\eta_j$ and 
where $b_i, a_{ij}$ are the coefficients of the  Runge-Kutta method.
Of course, other higher-order methods can be adapted to this framework, as for instance, variational Runge-Kutta-Munthe-Kaas integrators \cite{BoMa}
taking 
\[
{\mathfrak l}^{RKMK, \tau}_h(g)=h\sum_{i=1}^s b_i {\mathfrak l}(d_l\tau_{\xi_i}(\eta_i)))
\]
with  $\xi_i=h\sum_{j=1}^s a_{ij}d_l\tau_{\xi_j}^{-1}(\eta_j)$ and where $\eta_i$ are chosen to extremize ${\mathfrak l}_h^{RKMK, \tau}$  under the constraint  $\tau^{-1}(g)=h\sum_{j=1}^s b_jd_l\tau_{\xi_j}^{-1}(\eta_j)$.

\subsubsection{Lie groups of matrices}
In most of the examples (which come from Mechanics)  $G$ is a Lie subgroup of $ GL(n, {\mathbb R})$. Consider the corresponding Lie subalgebra ${\mathfrak g}\subseteq gl(n, {\mathbb R})$ and ${\mathfrak l}: {\mathfrak g}\rightarrow {\mathbb R}$ a regular Lagrangian admitting a regular extension ${\mathfrak l}_{ext}: {gl}(n, {\mathbb R}) \rightarrow {\mathbb R}$. 
In the most typical situation, we start from an inner product  $\langle \; ,\; \rangle_{\mathfrak g}: {\mathfrak g}\times {\mathfrak g}\rightarrow {\mathbb R}$ on the Lie algebra ${\mathfrak g}$ and the Lagrangian ${\mathfrak l}$ is given by: 
\[
{\mathfrak l}(\xi)=\frac{1}{2}\langle \xi, \xi\rangle_{\mathfrak g}, \; \; \; \mbox{ for } \xi \in {\mathfrak g}.
\]
Taking an arbitrary vector subspace ${\mathfrak h}$ such that 
 ${gl}(n, {\mathbb R})={\mathfrak g}\oplus {\mathfrak h}$, equipped  with an arbitrary inner product $\langle\; ,\; \rangle_{\mathfrak h}$, we decompose any element $a\in {gl}(n, {\mathbb R})$ as
 $a=a_1+a_2$ where $a_1\in {\mathfrak g}$ and $a_2\in {\mathfrak h}$. A regular extension of ${\mathfrak l}$ is 
 \[
 {\mathfrak l}_{ext}(a)= \frac{1}{2}\langle a_1, a_1\rangle_{\mathfrak g}+ \frac{1}{2}\langle a_2, a_2\rangle_{\mathfrak h}
\]

Given a regular extension ${\mathfrak l}_{ext}: {gl}(n, {\mathbb R}) \rightarrow {\mathbb R}$ it is easy to derive different discretizations. For instance, consider the map: 
\[
\begin{array}{rrcl}
\tau:&  {gl}(n, {\mathbb R})&\longrightarrow&  {Gl}(n, {\mathbb R})\\
        & a&\longmapsto&I+ha\; .
\end{array}
\]
Note that $\tau$ is a diffeomorphism from an open neighborhood of the zero matrix in  ${gl}(n, {\mathbb R})$ on an open neighborhood of the identity matrix $I$ in ${Gl}(n, {\mathbb R})$. The inverse of $\tau$ is given by $\tau^{-1}(A)=\frac{A-I}{h}$.  
Using the discretization (\ref{rew}) for the lagrangian ${\mathfrak l}_{ext}$ we have that
\[
\begin{array}{rrcl}
({\mathfrak l}_{ext})^{d, \tau, \alpha}_h:& {Gl}(n, {\mathbb R})&\longrightarrow& {\mathbb R}\\
                                    &A&\longmapsto &h{\mathfrak l}_{ext}(((1-\alpha)I+\alpha A)^{-1}\left(\displaystyle\frac{A-I}{h}\right))
\end{array}
\]                                    
Now, it is only necessary to consider the restriction of this discrete Lagrangian to $G$ to derive a discretization of ${\mathfrak l}: {\mathfrak g}\rightarrow {\mathbb R}$, that is, 
${\mathfrak l}^{d, \tau, \alpha}_h: G\rightarrow \R$ is defined by
\[
{\mathfrak l}^{d, \tau, \alpha}_h(A)=({\mathfrak l}_{ext})^{d, \tau, \alpha}_h(A), \hbox{  with  }  A\in G\; .
\]
Observe that, in general, $((1-\alpha)I+\alpha A)^{-1}\left(\displaystyle\frac{A-I}{h}\right)\notin {\mathfrak g}$ and then it is not possible to use the lagrangian ${\mathfrak l}: {\mathfrak g}\rightarrow {\mathbb R}$ instead of its extension ${\mathfrak l}_{ext}$. 

The order of the method is directly derived from the order of the extension $({\mathfrak l}_{ext})^{d, \tau, \alpha}_h$. For instance, for $\alpha=1/2$ we derive a second order method using  
\[
{\mathfrak l}^{d, \tau, 1/2}_h(A)=h{\mathfrak l}_{ext}(\left(\frac{I+A}{2}\right)^{-1}\left(\displaystyle\frac{A-I}{h}\right))\;.
\]
An alternative second order method is derived using 
\[
{\mathfrak l}_h^{sym,\tau, 0}(A) =\frac{h}{2}\left({\mathfrak l}_{ext}(\displaystyle\frac{A-I}{h}) +{\mathfrak l}_{ext}(A^{-1}\left(\displaystyle\frac{A-I}{h}\right))\right)\; .
\]

\subsection{Example: Lagrangians defined on a trivial principal bundle. Discretization of the Lagrange-Poincar\'e equations}

Let $K$ be a Lie group with Lie algebra ${\mathfrak k}$ and $M$ a smooth manifold. Then, we can consider the tangent lift of the standard left action of $K$ on $K \times M$. As we know, the space of orbits of this action of $K$ on $T(K \times M)$ is just the Atiyah algebroid associated with the trivial principal $K$-bundle $pr_2: K \times M \to M$, that is, 
\[
T(K \times M)/K \simeq {\mathfrak k} \times TM \to M.
\]
Note that using left-trivialization, we can identify the tangent bundle $T(K \times M)$ with the space $(K \times {\mathfrak k}) \times TM$ and, under this identification, the action of $K$ on $T(K \times M)$ is given by
\[
k ((k', \xi), (x^{i}, \dot{x}^{i})) = ((k k', \xi), (x^{i}, \dot{x}^{i})),
\]
for $k \in K$, $(k', \xi) \in K \times {\mathfrak k}$ and $(x^{i}, \dot{x}^{i}) \in TM$. So, it is clear that the space of orbits $T(K \times M)/K$ of this action is diffeomorphic to the product ${\mathfrak k} \times TM$.

Moreover, if $(\xi, X), (\eta, Y) \in {\mathfrak k} \times {\mathfrak X}(M)$ are projectable sections of the vector bundle ${\mathfrak k} \times TM \to M$ then the Lie bracket of $(\xi, X)$ and $(\eta, Y)$ is given by
\[
\lcf (\xi, X), (\eta, Y) \rcf = ([\xi, \eta]_{\mathfrak k}, [X, Y]),
\]
where $[\cdot, \cdot]_{\mathfrak k}$ is the Lie bracket in ${\mathfrak k}$ and $[\cdot, \cdot]$ is the standard Lie bracket of vector fields on $M$.

In addition, the anchor map in ${\mathfrak k} \times TM \to M$ is defined by
\[
\rho(\xi, X) = X.
\]
In this case, the product manifold $G = H \times (M \times M)$ is a Lie groupoid over the manifold $M$ and its Lie algebroid is just the Atiyah algebroid ${\mathfrak k} \times TM \to M$. In fact, the Lie groupoid structure on $G \times (M \times M)$ is just "the product" of the Lie group structure on $G$ and the pair Lie groupoid structure on $M \times M$.

Now, suppose that $L: T(K \times M) \simeq (K \times {\mathfrak k}) \times TM  \to \mathbb{R}$ is a left-invariant regular Lagrangian function on $T(K \times M)$. This means that
\[
L(g, \xi, x^{i}, \dot{x}^{i}) = L({\mathfrak e}, \xi, x^{i}, \dot{x}^{i}),
\]
where ${\mathfrak e}$ is the identity element in $K$.

Thus, $L$ defines a reduced regular Lagrangian function ${\mathfrak l}$ on the Atiyah algebroid
\[
{\mathfrak l}: {\mathfrak k} \times TM \to \mathbb{R}, \; \; \; (\xi, x^{i}, \dot{x}^{i}) \to {\mathfrak l}(\xi, x^{i}, \dot{x}^{i}) = L({\mathfrak e}, \xi, x^{i}, \dot{x}^{i}).
\]
The corresponding Euler-Lagrange equations for this function are the well-known Lagrange-Poincar\'e equations (see \cite{CeMaRa}): 
   \begin{equation}\label{lpe}
   \frac{d}{dt}\left(\frac{\partial {\mathfrak l}}{\partial \dot{x}^i}\right)-\frac{\partial {\mathfrak l}}{\partial x^i}=0, \qquad \frac{d}{dt}\left(\frac{\partial {\mathfrak l}}{\partial \xi}\right)=\hbox{ad}^*_{\xi}\frac{\partial {\mathfrak l}}{\partial \xi}, \qquad \dot{x}^{i} = \frac{dx^{i}}{dt}.
 \end{equation}
 It is well-known that if $(\xi, x): I \subseteq \mathbb{R} \to {\mathfrak k} \times M$ is a solution of the Lagrange-Poincar\'e equations then there exists a unique solution
 \[
 t \to (k(t), x^{i}(t))
\]
of the Euler-Lagrange equations for $L$ satisfying
\[
k(0) = {\mathfrak e} \; \; \mbox{ and } \dot{k}(t) = k(t) \xi(t) = (T_{\mathfrak e} l_{k(t)})(\xi(t))
\]
(see \cite{CeMaRa}).

Now, for a point $x \in M$, using the results in Section \ref{exact-discrete-Lie-groupoid}, we have that there is a sufficiently small positive number $h > 0$, an open neighborhood $V$ of ${\mathfrak e}$ and an open neighborhood $W$ of $x$, such that the exact discrete Lagrangian function ${\mathfrak l}^{e}_{h}: U = V \times (W \times W) \subseteq K \times (M \times M) \to \mathbb{R}$ is given by
\[
{\mathfrak l}^{e}_{h}(k, x_0, x_1) = \int_{0}^{h} {\mathfrak l}((\xi, x)_{(k, x_0,x_1)}(t))dt
\]
where $(\xi, x)_{(k, x_0,x_1)}: I \subseteq \mathbb{R} \to {\mathfrak k} \times M$ is the unique solution of the Lagrange-Poincar\'e equations for ${\mathfrak l}: {\mathfrak k} \times TM \to \mathbb{R}$ such that the corresponding solution of the Euler-Lagrange equations for $L$ satisfies
\[
k(0) = {\mathfrak e}, \; \; k(h) = k, \; \; x(0) = x_0, \; \; x(h) = x_1.
\]
Next, as in Example \ref{example1}, our aim is to derive a discretization of ${\mathfrak l}_{h}^{e}$. We will use an auxiliar Riemannian metric ${\mathcal G}$ on $M$, with associated geodesic spray $\Gamma_{\mathcal G}$, and for a sufficiently small positive number $h_0 > 0$ we will denote by $exp_{h_0}^{\Gamma_{\mathcal G}}$ the corresponding associated exponential map to $\Gamma_{\mathcal G}$ and by $\widetilde{exp}^{\Gamma_{\mathcal G}}_{h_0}$ the map given by (\ref{exp-tilde-h0}). 

So,  a discretization of ${\mathfrak l}^e_h$ would be
 \[
{\mathfrak l}^d_h(k, x_0, x_1)=h {\mathfrak l}(\frac{{exp}^{-1}(k)}{h}, (\widetilde{exp}^{\Gamma_{\mathcal G}}_{h})^{-1}(x_0, x_1))\; .
 \]
 In the particular case when $M={\mathbb R}^n$ and ${\mathcal G}$ is the Euclidean metric on $\mathbb{R}^n$, we have that 
  \[
{\mathfrak l}^d_h(k, x_0, x_1)=h {\mathfrak l}( \frac{{exp}^{-1}(k)}{h}, \frac{x_0+x_1}{2}, \frac{x_1-x_0}{h})\; .
 \]

\section{{ Conclusions and future work}}

{ In this paper, we obtained the exact discrete Lagrangian function for a regular continuous Lagrangian function which is defined on the total space of a Lie algebroid. For this purpose, we discussed convexity theorems for second order differential equations on Lie algebroids.}

{ Lagrangian systems on Lie algebroids appear, in a natural way, after the reduction of standard Lagrangian systems which are invariant under the action of a symmetry Lie group. So, our results are interesting for the discussion of the variational error analysis associated with discrete Lagrangian systems which are obtained after the reduction of symmetric Lagrangian systems. In particular, our results are applied to the discretization of the Euler-Poincar\'e and Lagrange-Poincar\'e equations.}

{ Anyway, more work about these topics must be done and, in some future papers, we are planning to address the following problems:}
 
$\bullet$ {\bf Discrete vakonomic mechanics and optimal control theory.} { In some recent papers, some authors \cite{GrGr,IMMD} have discussed the extension of the classical constrained variational calculus (also called vakonomic mechanics) to Lie algebroids using an intrinsic formulation which completely clarifies the geometry of these systems.  They have interesting applications to subriemannian geometry, optimal control theory etc.
The geometry of these  systems is usually described by two data, a Lagrangian $L: AG\to \mathbb{R}$ and a submanifold $N$ of $AG$. The dynamics is derived using variational procedures and the equations typically include the evolution of additional variables (momenta, Lagrange multipliers, costate variables, depending on the context). 
Therefore, we have a set of second order differential equations which are coupled with first order equations. Obviously, in this case, it is not possible to use the techniques developed in this paper without a deep adaptation. 

An interesting  possibility to explore in the future is to assume that the Lagrangian $L$ is regular. Using the associated \sode\ $\Gamma_L$, then locally we can discretize the submanifold $N$ to a submanifold 
${\mathbb{N}}^{e}_h$ of $G$ applying the  exponential map $exp^{\Gamma_L}_h$. Moreover, we have the  exact discrete Lagrangian function ${\mathbb{L}}^{e}_h: G \to \mathbb{R}$ associated with $L: AG \to \mathbb{R}$. 
Now, with these two ingredients $({\mathbb{L}}^{e}_h, {\mathbb{N}}^{e}_h)$, it is possible to derive, applying discrete constrained variational calculus \cite{MaMaSt}, a discrete evolution operator.  We will study in a future paper the possible relation of this discrete evolution operator with the continuous dynamics of the constrained variational system determined by $L$ and $N$. 
}

$\bullet$ {\bf Discrete non-holonomic mechanics}. {Non-holonomic mechanical systems on Lie algebroids have been discussed, very recently, in several papers (see, for instance, \cite{CoLeMaMa,GrLeMaMa,LeMaMa1}). In this setting, a continuous regular non-holonomic Lagrangian system on a Lie algebroid $A$ over a manifold $M$ is a triple $(A, L, D)$, where $L:A \to \mathbb{R}$ is a regular Lagrangian function and $D$ is a vector subbundle of $A$ over $M$, the constraint distribution, which is not a Lie subalgebroid. The non-holonomic dynamics is given by a second order differential equation $\Gamma_{L, D}$ along $D$.}

{ On the other hand, a regular discrete non-holonomic Lagrangian system on a Lie groupoid $G$ over $M$ is determined by: i) a regular discrete Lagrangian function $L_d: G \to \mathbb{R}$ on $G$; ii) a constraint bundle $D_c$, which is a vector subbundle of the Lie algebroid $AG$ of $G$ and iii) a discrete constraint embedded submanifold ${\mathcal M}_d$ of $G$ such that $dim ({\mathcal M}_d) = dim (D_c)$. A solution of the discrete non-holonomic system is a sequence in ${\mathcal M}_d$ which satisfies the corresponding discrete Holder's principle (see \cite{CoMa,IgMaMaMa}).}

{ As in the unconstrained case and for the variational error analysis, it would be interesting to associate with every regular continuous non-holonomic Lagrangian system $(G, L, D)$ on $AG$ a regular exact discrete non-holonomic Lagrangian system on $G$. It is clear that the discrete constraint distribution of this system will be just $D$ and, using the results in this paper, we can consider the exact regular discrete Lagrangian function ${\mathbb{L}}^{e}_h: G \to \mathbb{R}$ associated with $L: AG \to \mathbb{R}$.}  

{ Now, in order to construct the discrete constraint submanifold, we could proceed as follows. First of all, we should extend the convexity theorems in this paper for the more general case of second order differential equations which are only defined along vector subbundles of $AG$. Then, the discrete constraint submanifold ${\mathcal M}_h^{e}$ will be the image in $G$ of the exponential map $exp^{\Gamma_{L, D}}_h$ associated with the constrained second order differential equation $\Gamma_{L,D}$. So, $D$ and ${\mathcal M}_h^{e}$ will be diffeomorphic.}

{ Finally, we must check that, under the previous diffeomorphism, the evolution of this discrete system for a sufficiently small positive number $h > 0$ is just the flow of $\Gamma_{L, D}$ at time $h$.} 
 
\appendix
\section{Proof of Theorem \ref{convexity1}} 
\label{Hartmann}

In this appendix, we will give a proof of Theorem \ref{convexity1}.
For this purpose, we will use some standard results on second order differential equations on $\R^n$ (see \cite{Ha}).

Let
\[
\displaystyle \frac{d^2q^{i}}{dt^2} = \xi^{i}(q^j, \frac{dq^j}{dt}), \; \; \forall i \in \{1, \dots, n\}
\]
be a system of second order differential equations on $\R^n$, with $\xi^{i}$ a real function on an open subset of $\R^{2n}$.

We will consider the problem of the existence of solutions satisfying the boundary conditions
\[
q^{i}(0) = 0, \; \; \; q^{i}(h_0) = q^{i}_{0}, \; \; \forall i, \mbox{ with } h_0 > 0.
\]
Then, we have the following result
\begin{theorem}\label{Hartman}(see Corollary 4.1 of  Chapter XII in \cite{Ha}) 
Let $\xi^{i}(q, \dot{q})$ be continuous for $1 \leq i \leq n$, $\|q\| \leq R_0$, $\|\dot{q}\| \leq R_1$
such that $f$ satisfies a Lipschitz condition with respect to $q, \dot{q}$ of the form
\[
\| \xi(q_1^j, \dot{q}^j_1) - \xi(q_2^j, \dot{q}_2^j)\| \leq \theta_1 \| q_2 - q_1\| + \theta_2 \| \dot{q}_2 - \dot{q}_1 \|
\]
with Lipschitz constants $\theta_1, \theta_2$, so small that
\[
\frac{\theta_1 h_0^2}{8} + \frac{\theta_2 h_0}{2} < 1.
\]
In addition, suppose that $\| \xi(q^j, \dot{q}^j) \| \leq M$ and that
\[
\frac{M h_0^2}{8} + \| q_0\| \leq R_0, \; \; \; \frac{M h_0}{2} + \frac{\| q_0\|}{h_0} \leq R_1.
\]
Then, the system of second order differential equations
\[
\frac{d^2 q^{j}}{dt^2} = \xi^{j} (q^{i}, \dot{q}^{i}), \; \; \mbox{ for all } j
\]
has a unique solution satisfying
\[
q^{i}(0) = 0, \; \; \mbox{ and } q^{i}(h_0) = q_{0}^{i}, \; \; \mbox{ for all } i.
\]
\end{theorem}

Now, we may prove Theorem \ref{convexity1}.

\begin{proof}({\it proof of Theorem \ref{convexity1}})

Let $(\tilde{U}, \tilde{\varphi} \equiv (q^i))$ be a local chart
on $Q$ such that
\[
\tilde{\varphi}(\tilde{U}) = B(0; \epsilon) \; \mbox{ and }
\tilde{\varphi}(q_0) = (0, \dots , 0),
\]
where $B(0; \epsilon)$ is the open ball in $\mathbb{R}^n$ of
center the origin and radius $\epsilon > 0$.

We consider the corresponding local coordinates
$(\tau_{TQ}^{-1}(\tilde{U}), \bar{\varphi} \equiv (q^i, v^i))$ on
$TQ$. Note that $\bar{\varphi}(\tau_{TQ}^{-1}(\tilde{U})) =
\tilde{\varphi}(\tilde{U}) \times \mathbb{R}^n$. Since $\Gamma$ is
a \sode, we also have that
\[
\Gamma = \displaystyle v^i \frac{\partial}{\partial q^i} + \xi^{i}
(q, v) \frac{\partial}{\partial v^i}.
\]
Then, the trajectories of $\Gamma$ in $\tilde{U}$ are the
solutions of the system of second order differential equations
\[
\displaystyle \frac{d^2q^i}{dt^2} = \xi^i(q, \frac{dq}{dt}), \; \;
\; \mbox{ for all } i.
\]
Now, using that $\xi^i \in C^{\infty}(B(0; \epsilon) \times
\mathbb{R}^n)$, we deduce that it is possible to choose $R_0, R_1 > 0$, 
$\theta_0, \theta_1 \in \R$, $h_0 > 0$ and $M \in \R$ as in Theorem \ref{Hartman}.
Thus, if we take the open subset $U$ of $Q$ defined by $U =
\tilde{\varphi}^{-1}(B(0; R))$, with $R = min (R_{0} -
\displaystyle \frac{Mh_{0}^2}{8}, h_{0}R_{1} - \displaystyle
\frac{Mh_{0}^{2}}{2})$, we conclude that for
every $q_{1} \in U$ there exists a unique trajectory
$\sigma_{q_{0}q_{1}}: [0, h_{0}] \to \tilde{U} \subseteq Q$ of
$\Gamma$ such that
\[
\sigma_{q_{0}q_{1}}(0) = q_0, \; \; \; \sigma_{q_{0}q_{1}}(h_0) =
q_1.
\]
This ends the proof of the result.
\end{proof}

\section{Lie algebroids and groupoids}\label{algebroide-grupoide}

First of all, we will recall the definition of a Lie groupoid
and some generalities about them are explained (for more details,
see \cite{Mac}).

A groupoid over a set $M$ is a set $G$ together with the
following structural maps:
\begin{itemize}
\item A pair of maps $\alpha: G \to M$, the source, and
$\beta: G \to M$, the target. Thus, an element $g \in G$ is
thought as an arrow from $x= \alpha(g)$ to $y = \beta(g)$ in $M$
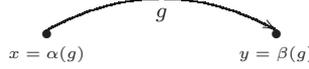
$$
\xymatrix{*=0{\stackrel{\bullet}{\mbox{\tiny
 $x=\alpha(g)$}}}{\ar@/^1pc/@<1ex>[rrr]_g}&&&*=0{\stackrel{\bullet}{\mbox{\tiny
$y=\beta(g)$}}}}
$$
The maps $\alpha$ and $\beta$ define the set of composable pairs
$$
G_{2}=\{(g,h) \in G \times G / \beta(g)=\alpha(h)\}.
$$
\item A multiplication $m: G_{2} \to G$, to be denoted simply by $m(g,h)=gh$, such that
\begin{itemize}
\item $\alpha(gh)=\alpha(g)$ and $\beta(gh)=\beta(h)$.
\item $g(hk)=(gh)k$.
\end{itemize}
If $g$ is an arrow from $x = \alpha(g)$ to $y = \beta(g)
$ and $h$ is an arrow from $y=\beta(g) = \alpha(h)$ to $z = \beta(h)$ then
$gh$ is the composite arrow from $x$ to $z$
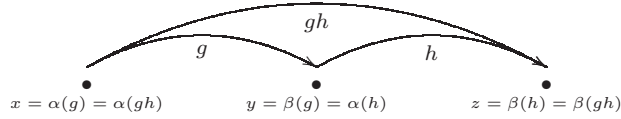
$$\xymatrix{*=0{\stackrel{\bullet}{\mbox{\tiny
 $x=\alpha(g)=\alpha(gh)$}}}{\ar@/^2pc/@<2ex>[rrrrrr]_{gh}}{\ar@/^1pc/@<2ex>[rrr]_g}&&&*=0{\stackrel{\bullet}{\mbox{\tiny
 $y=\beta(g)=\alpha(h)$}}}{\ar@/^1pc/@<2ex>[rrr]_h}&&&*=0{\stackrel{\bullet}{\mbox{\tiny
 $z=\beta(h)=\beta(gh)$}}}}$$
\item An identity map $\varepsilon: M \to G$, a section of $\alpha$ and $\beta$, such that
\begin{itemize}
\item $\varepsilon(\alpha(g))g=g$ and $g\varepsilon(\beta(g))=g$.
\end{itemize}
\item An inversion map $i: G \to G$, to be denoted simply by $i(g)=g^{-1}$, such that
\begin{itemize}
\item $g^{-1}g=\varepsilon(\beta(g))$ and $gg^{-1}=\varepsilon(\alpha(g))$.
\end{itemize}
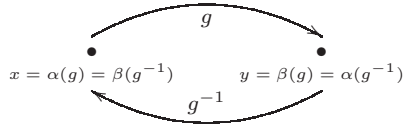
$$\xymatrix{*=0{\stackrel{\bullet}{\mbox{\tiny
 $x=\alpha(g)=\beta(g^{-1})$}}}{\ar@/^1pc/@<2ex>[rrr]_g}&&&*=0{\stackrel{\bullet}{\mbox{\tiny
 $y=\beta(g)=\alpha(g^{-1})$}}}{\ar@/^1pc/@<2ex>[lll]_{g^{-1}}}}$$

\end{itemize}

A groupoid $G$ over a set $M$ will be denoted simply by the symbol
$G \rightrightarrows M$.

The groupoid $G \rightrightarrows M$ is said to be a Lie
groupoid if $G$ and $M$ are manifolds and all the structural maps
are differentiable with $\alpha$ and $\beta$ differentiable
submersions. If $G \rightrightarrows M$ is a Lie groupoid then $m$
is a submersion, $\varepsilon$ is an immersion and $i$ is a
diffeomorphism. Moreover, if $x \in M$, $\alpha^{-1}(x)$ (resp.,
$\beta^{-1}(x)$) will be said the $\alpha$-fiber (resp.,
the $\beta$-fiber) of $x$.

Typical examples of Lie groupoids are: the pair or banal groupoid $Q \times Q$ over $Q$, a Lie group $G$ (as a Lie groupoid over a single point), the Atiyah groupoid $(Q \times Q)/G$ (over $Q/G$) associated with a free and proper action of a Lie group $G$ on $Q$ and the Lie groupoid $G\pi$ associated with a fibration $\pi: P \to M$ given by
\[
G\pi = \{(p, p') \in P \times P / \pi(p) = \pi(p') \},
\]
(it is a Lie subgroupoid of the pair groupoid $P \times P \rightrightarrows P$).    

On the other hand, if $G \rightrightarrows M$ is a Lie groupoid and $g \in G$ then the left-translation by
$g \in G$ and the right-translation by $g$ are the
diffeomorphisms
$$
\begin{array}{lll}
l_{g}: \alpha^{-1}(\beta(g)) \longrightarrow
\alpha^{-1}(\alpha(g))&; \; \;& h \longrightarrow
l_{g}(h) = gh, \\
r_{g}: \beta^{-1}(\alpha(g)) \longrightarrow
\beta^{-1}(\beta(g))&; \; \;& h \longrightarrow r_{g}(h) = hg.
\end{array}
$$
Note that $l_{g}^{-1} = l_{g^{-1}}$ and $r_{g}^{-1} = r_{g^{-1}}$.

A vector field $\tilde{X}$ on $G$ is said to be
left-invariant (resp., right-invariant) if it is
tangent to the fibers of $\alpha$ (resp., $\beta$) and
$\tilde{X}(gh) = (T_{h}l_{g})(\tilde{X}(h))$ (resp.,
$\tilde{X}(gh) = (T_{g}r_{h})(\tilde{X}(g)))$, for $(g,h) \in
G_{2}$.

Now, we will recall the definition of the Lie algebroid
associated with $G$.

We consider the vector bundle $\tau: AG \to M$, whose fiber at a
point $x \in M$ is $A_{x}G = V_{\varepsilon(x)}\alpha = Ker
(T_{\varepsilon(x)}\alpha)$. It is easy to prove that there exists a
bijection between the space $\Gamma (\tau)$ of sections of $\tau: AG \to M$ and the set of
left-invariant (resp., right-invariant) vector fields on $G$. If
$X$ is a section of $\tau: AG \to M$, the corresponding
left-invariant (resp., right-invariant) vector field on $G$ will
be denoted by $\lvec{X}$ (resp., $\rvec{X}$), where
\begin{equation}\label{linv}
\lvec{X}(g) = (T_{\varepsilon(\beta(g))}l_{g})(X(\beta(g))),
\end{equation}
\begin{equation}\label{rinv}
\rvec{X}(g) = -(T_{\varepsilon(\alpha(g))}r_{g})((T_{\varepsilon
(\alpha(g))}i)( X(\alpha(g)))),
\end{equation}
for $g \in G$. Using the above facts, one may introduce a 
bracket $\lcf\cdot , \cdot\rcf$ on the space of sections $\Gamma(AG)$ and a bundle map $\rho: AG \to TM$, which
are defined by
\begin{equation}\label{LA}
\lvec{\lcf X, Y\rcf} = [\lvec{X}, \lvec{Y}], \makebox[.3cm]{}
\rho(X)(x) = (T_{\varepsilon(x)}\beta)(X(x)),
\end{equation}
for $X, Y \in \Gamma(AG)$ and $x \in M$. 

Since $[\cdot, \cdot]$ induces a Lie algebra structure on the space of vector fields on $G$, it is easy to prove that $\lcf \cdot, \cdot \rcf$ also defines a Lie algebra structure on $\Gamma(AG)$. In addition, it follows that
\[
\lcf X, f Y\rcf = f \lcf X, Y\rcf + \rho(X)(f) Y,
\]
for $X, Y \in \Gamma(AG)$ and $f \in C^{\infty}(M)$.

In other words, we have a Lie algebroid structure on the vector bundle $\tau: AG \to M$ with anchor map $\rho$.
It is the Lie algebroid of $G$.

We remark the following facts:
\begin{itemize} 
\item
The Lie algebroid of the pair groupoid $Q \times Q$ over $Q$ is the standard Lie algebroid $TQ \to Q$;

\item
The Lie algebroid of a Lie group $G$ is the Lie algebra ${\mathfrak g}$ of $G$;

\item
The Lie algebroid of the Atiyah groupoid $(Q \times Q)/G$ is the Atiyah algebroid $TQ/G$ over $Q/G$ and

\item
The Lie algebroid of the Lie groupoid $G\pi$ associated with a fibration $\pi: P \to M$ is the vertical bundle $V\pi$ of $\pi: P \to M$.  
\end{itemize}

\end{document}